\documentclass{article}
\usepackage{bbm}
\usepackage{caption}
\usepackage{rotating}
\usepackage[maxfloats=100]{morefloats}
\usepackage{listings}
\usepackage{booktabs}
\usepackage{xcolor}
\usepackage[title]{appendix}
\usepackage{makecell} 
\usepackage{multirow}

\usepackage[symbol]{footmisc}

\usepackage{longtable}
\usepackage{subfigure}
\usepackage[numbers]{natbib}
\usepackage{geometry}

\usepackage{verbatim}
\usepackage{xcolor}
\usepackage{graphicx}
\usepackage{enumitem}
\usepackage{comment}
\usepackage{subfiles}

\usepackage{todonotes}
\usepackage{tikz-qtree,tikz-qtree-compat}
\usepackage{scrextend}
\usepackage{framed}
\usepackage{placeins}
\usepackage{setspace}
\usepackage{algorithmic}
\usepackage{algorithm}
\usepackage{mathtools,collcell,eqparbox}

\usepackage{multicol}
\usepackage{pgfplots}
\usepackage{tikz}

\usetikzlibrary{shapes,decorations,arrows,calc,arrows.meta,fit,positioning}
\tikzset{
	-Latex,auto,node distance =1 cm and 1 cm,semithick,
	state/.style ={ellipse, draw, minimum width = 0.7 cm},
	point/.style = {circle, draw, inner sep=0.04cm,fill,node contents={}},
	bidirected/.style={Latex-Latex,dashed},
	el/.style = {inner sep=2pt, align=left, sloped}
}
\usetikzlibrary{matrix,decorations.pathreplacing}
\usepackage{listings}
\usepackage{courier}
\usepackage{url}
\usepackage{color}
\usepackage[title]{appendix}
\usepackage{authblk}

\usepackage{amssymb}
\usepackage{amsthm}
\usepackage{ dsfont }
\usepackage{amsmath}

\usepackage[pagebackref = true]{hyperref}
\hypersetup{
    colorlinks = true,
    linkcolor = teal,
    anchorcolor = teal,
    citecolor = teal,
    filecolor = teal,
    urlcolor = teal
    }

\pgfplotsset{compat=1.15}
\newcounter{BMatrix}

\newcommand{\setmaxwd}[1]{%
  \eqmakebox[BM-\theBMatrix][\BMalign]{$#1$}%
}
\MHInternalSyntaxOn

\MHInternalSyntaxOff
\makeatother

\newtheorem{theorem}{Theorem}

\newtheorem{proposition}{Proposition}[section]

\newtheorem{corollary}{Corollary}[theorem]
\newtheorem{lemma}{Lemma}

\theoremstyle{definition}
\newtheorem{definition}{Definition}

\theoremstyle{remark}
\newtheorem{remark}{Remark}
\newtheorem*{remark*}{Remark}

\newcommand{\inv}{^{-1}}

\newcommand{\1}{\mathds{1}}

\newcommand{\R}{\mathbb{R}}

\newcommand{\Z}{\mathbb{Z}}
\newcommand{\N}{\mathbb{N}}

\newcommand{\ra}{\rightarrow}

\newcommand{\la}{\leftarrow}

\newcommand{\cd}{\cdot}
\newcommand{\ds}{\cdots}

\newcommand{\mrm}[1]{\mathrm{#1}}
\newcommand{\KH}{\mathrm{K}_{\mathrm{H}}}

\newcommand{\KS}{\mathrm{K}_{\mathrm{S}}}

\newcommand{\PiH}{\Pi_{\mathrm{H}}}

\newcommand{\PiS}{\Pi_{\mathrm{S}}}

\newcommand{\cA}{\mathcal{A}}

\newcommand{\cF}{\mathcal{F}}
\newcommand{\cG}{\mathcal{G}}
\newcommand{\cH}{\mathcal{H}}

\newcommand{\cP}{\mathcal{P}}
\newcommand{\cQ}{\mathcal{Q}}
\newcommand{\cR}{\mathcal{R}}
\newcommand{\cS}{\mathcal{S}}
\newcommand{\cT}{\mathcal{T}}

\newcommand{\spnorm}[1]{\left|#1\right|_\mathrm{span}}
\newcommand{\set}[1]{\left\{{#1}\right\}}

\newcommand{\norminf}[1]{\left\|#1\right\|_{\infty}}

\newcommand{\abs}[1]{\left|#1\right|}
\newcommand{\sqbk}[1]{\left[ #1 \right]}
\newcommand{\sqbkcond}[2]{\left[ #1 \middle| #2 \right]}

\newcommand{\crbk}[1]{\left( #1 \right)}

\newcommand{\bd}[1]{\mathbf{#1}}

\DeclareMathOperator*{\argmax}{arg\,max}
\DeclareMathOperator*{\argmin}{arg\,min}

\definecolor{azure}{rgb}{0.0, 0.4, 0.9}

\definecolor{darkred}{rgb}{0.6, 0, 0}

\definecolor{codegreen}{rgb}{0,0.4,0}
\definecolor{codeblue}{rgb}{0.1,0.1,0.7}
\definecolor{codegray}{rgb}{0.5,0.5,0.5}
\definecolor{codepurple}{rgb}{0.58,0,0.82}
\definecolor{backcolour}{rgb}{0.97,0.97,0.97}
 
\lstdefinestyle{mystyle}{
    backgroundcolor=\color{backcolour},   
    commentstyle=\color{codegreen},
    keywordstyle=\color{magenta},
    numberstyle=\tiny\color{codegray},
    stringstyle=\color{codepurple},
    basicstyle=\scriptsize\ttfamily,
    identifierstyle=\color{codeblue},
    breakatwhitespace=false,         
    breaklines=true,                 
    captionpos=b,                    
    keepspaces=true,                 
    numbers=left,                    
    numbersep=4pt,                  
    showspaces=false,                
    showstringspaces=false,
    showtabs=true,                  
    tabsize=3
}
\numberwithin{equation}{section}

\title{Bellman Optimality of Average-Reward Robust Markov Decision Processes with a Constant Gain}

\author[1]{Shengbo Wang}
\affil[1]{Daniel J. Epstein Department of Industrial and Systems Engineering\\
University of Southern California}

\author[2]{Nian Si}
\affil[2]{Department of Industrial Engineering and Decision Analytics\\
Hong Kong University of Science and Technology}
\date{October 2025}

\begin{document}
\maketitle
\begin{abstract}
Learning and optimal control under robust Markov decision processes (MDPs) have received increasing attention, yet most existing theory, algorithms, and applications focus on finite-horizon or discounted models. Long-run average-reward formulations, while natural in many operations research and management contexts, remain underexplored. This is primarily because the dynamic programming foundations are technically challenging and only partially understood, with several fundamental questions remaining open. This paper steps toward a general framework for average-reward robust MDPs by analyzing the constant-gain setting. We study the average-reward robust control problem with possible information asymmetries between the controller and an S-rectangular adversary. Our analysis centers on the constant-gain robust Bellman equation, examining both the existence of solutions and their relationship to the optimal average reward. Specifically, we identify when solutions to the robust Bellman equation characterize the optimal average reward and stationary policies, and we provide one-sided weak communication conditions ensuring solutions' existence. These findings expand the dynamic programming theory for average-reward robust MDPs and lay a foundation for robust dynamic decision making under long-run average criteria in operational environments.
\end{abstract}


\section{Introduction}

Markov Decision Processes (MDPs) provide a foundational framework for modeling sequential decision-making under uncertainty, underpinning much of modern data-driven dynamic decision-making and reinforcement learning (RL) \citep{sutton2018reinforcement}. MDPs have long been central to operations research, including dynamic control of queuing systems, service allocation, admission control, and inventory replenishment \citep{wein1992dynamic,dai2005maximum,ward2008asymptotically,xin2022distributionally}. At the same time, data-driven stochastic control continues to advance and attract new research interest \citep{berberich2024model_predictive_ctrl,guo2025fastpolicylearninglinear}. In parallel, the past decade has witnessed remarkable successes of RL algorithms in increasingly sophisticated simulated environments, including superhuman performance in Atari games \citep{mnih2013playing}, mastery of Go \citep{silver2017mastering}, and progress toward AI reasoning agents \citep{guo2025deepseek}. 

Nevertheless, generalizability and robustness of these methods to out-of-sample, real-world environments remain limited, owing to model misspecification that can arise from discrepancies in dynamics, partial observability, stochasticity, and unaccounted real-world perturbations. To bridge this gap and enhance policy reliability in practical deployment, the robust MDP framework has emerged as a principled approach, explicitly accounting for model ambiguity and worst-case environment disturbances while still preserving (in most cases) the tractability of MDP models \citep{iyengar2005robust,nilim2005robust,mannor2016robust,wiesemann2013robust,wang2023foundationRMDP}. This robustness perspective is particularly important in operations environments, where data are often limited and practitioners may value stability over optimality. In such settings, ambiguity in arrivals, service times, demands, and supply conditions can induce significant policy risks, while well-designed policies can restore stability and reliable performance.

While discounted-reward robust MDPs have been relatively well-studied, the average-reward setting introduces significant theoretical challenges that remain largely unresolved. In particular, the optimality conditions, commonly referred to as the Bellman optimality equations, have not yet been fully characterized for average-reward robust MDPs. This gap hinders the development of efficient algorithms, sample-efficient learning methods, and performance guarantees, motivating our investigation into these foundational open problems.

In this paper, we consider optimality conditions for average-reward robust Markov decision processes (MDPs). Specifically, we consider the finite-state and finite-action setting and define
$$\overline{\alpha}(\mu,\Pi,\mrm{K})
:= \sup_{\pi \in \Pi}\inf_{\kappa \in \mrm{K}}
\limsup_{n \ra \infty} 
E_\mu^{\pi,\kappa} \left[\frac{1}{n}\sum_{k=0}^{n-1} r(X_k,A_k)\right],$$
where $\Pi$ and $\mrm{K}$ denote the controller’s and adversary’s policy classes, respectively, $\mu$ is the initial distribution, and $r(\cdot,\cdot)\in[0,1]$ is a bounded reward function. We consider S-rectangular adversary policy classes, in which the adversary’s perturbation of the transition probabilities is state-wise separable: the choice at one state does not affect the set of admissible choices at any other state. However, action dependencies within each state may remain coupled, thereby enforcing constraints among the adversary’s choices for different actions at the same state. Moreover, both controller and adversary's policy classes may be history-dependent (denoted by $\PiH$ and $\KH$) or stationary (denoted by $\PiS$ and $\KS$), with potentially asymmetric information structures between the controller and adversary decisions.

We are interested in the constant-gain Bellman optimality, that is, identifying conditions under which  
$\overline{\alpha}(\mu,\Pi,\mrm{K}) = \alpha^*$ holds for all initial distributions $\mu$. Here, $\alpha^* \in [0,1]$ is part of a solution pair $(u^*, \alpha^*)$ to the following robust Bellman equation with a constant gain:
\begin{equation}\label{eqn:r_bellman_eqn_const_gain0}
    u^*(s) 
    = \sup_{\phi \in \cQ}\inf_{p_s \in \cP_s}
      E_{\phi,p_s}\!\left[r(s,A_0) - \alpha^* + u^*(X_1)\right],
\end{equation}
where  the expectation is taken w.r.t. the measure $P_{\phi,p_s}(A_0 = a,X_1 = s') = \phi(a)p_{s,a}(s')$.
Here, $\cQ \subseteq \mathcal{P}(\mathcal{A})$ denotes the controller’s admissible decision set at state $s$.  The introduction of $\cQ$ serves applications where additional restrictions on the policy structure are required. For example, in healthcare applications, randomized policies may be inappropriate due to fairness and ethical concerns, and one may need to restrict attention to deterministic actions. Here, we highlight two main types of decision sets:  
(i) the deterministic policy set $\cQ = \{\delta_a : a \in {A}\}$, where $\delta_a$ is the Dirac measure at action $a$; and 
(ii) the fully randomized action set $\cQ = \mathcal{P}(\mathcal{A})$.  
Moreover, $\cP_s$ denotes the adversary’s decision set, which can be understood as the projection of the ambiguity set $\cP$ onto state $s$.

In standard MDP settings, it is well understood that weak communication is sufficient for constant-gain Bellman optimality, and that a solution to the Bellman equation characterizes an optimal stationary \emph{deterministic} policy \citep{puterman2014MDP}. Extending these results to robust MDPs, however, presents significant theoretical challenges for several reasons:
\begin{enumerate}
    \item The relationship between the solution of the robust Bellman equation and the optimal robust control value is not straightforward. \citet{grand2023beyond} provides an example, adapted from the classical \emph{Big Match} game, showing that the optimal control value $\overline{\alpha}(\mu,\PiH,\mrm{K})$ can be strictly larger than the value obtained from any stationary policy.
    \item It is well known that robust MDPs may require randomized policies for optimal decision-making \citep{wiesemann2013robust}. When randomized policies are optimal, Blackwell optimality typically fails \citep{grand2023beyond}. This prevents a direct application of the standard Blackwell optimality argument used to establish Bellman optimality in classical average-reward MDPs.
    \item Weak communication-type assumptions are much harder to analyze under the controller-adversary dynamics. In particular, different stationary controller/adversary policies can induce different communicating classes. Hence, unlike classical MDPs, it is unclear a priori what the communicating class associated with an optimal policy should be for robust MDPs.
\end{enumerate}

In this paper, we first clarify the implications of a solution to the constant-gain robust Bellman equation \eqref{eqn:r_bellman_eqn_const_gain0}, including the extent to which it characterizes the optimal robust control value and the associated policies. When \eqref{eqn:r_bellman_eqn_const_gain0} does characterize the optimal robust control, we introduce one-sided (weak) communication conditions--natural generalizations of their classical MDP counterparts to robust MDPs--and show that they are sufficient to ensure the existence of a solution to \eqref{eqn:r_bellman_eqn_const_gain0}. Specifically, the controller is said to be (weakly) communicating if, for every stationary controller policy, the induced MDP faced by the adversary is (weakly) communicating. Likewise, the adversary is (weakly) communicating if every MDP within the ambiguity set $\cP$ is (weakly) communicating. For cases in which Bellman optimality fails due to information asymmetry, we provide a thorough treatment that yields a necessary and sufficient characterization of the control value. Our main technical results can be summarized as follows.

\begin{itemize}
    \item When the robust Bellman equation \eqref{eqn:r_bellman_eqn_const_gain0} admits a solution $(u^*,\alpha^*)$, $\alpha^*$ coincides with the optimal average-rewards $\alpha^* = \overline{\alpha}(\PiH,\KH) = \overline{\alpha}(\PiS,\KH) = \overline{\alpha}(\PiS,\KS)$, independent of the initial distribution (cf. Section \ref{subsec:bellmanOpt}). In particular, stationary controller policies are optimal for $\overline{\alpha}(\PiH,\KH)$. However, in general, $\overline{\alpha}(\PiH,\KS)\neq \alpha^*$; (cf. Section \ref{sec:HD_S}). 
    \item Theorem \ref{thm:policy_opt_gap}, together with Remark \ref{rmk:policy_from_bellman}, certifies the optimality of the policy derived from any solution of the robust Bellman equation.

    \item Theorem \ref{thm:wc_ctrl} shows that if the controller is weakly communicating (as in Definition \ref{def:wc}) and $\cQ$ and $\cP_s$ are \textit{compact} for all $s \in S$, then \eqref{eqn:r_bellman_eqn_const_gain0} has a solution.
    \item Theorem \ref{thm:wc_adv} and \ref{thm:suff_cond_for_swap} together imply that if the adversary is weakly communicating and $\cQ$ and $\cP_s$ are \textit{convex and compact} for all $s \in S$, then the  Bellman equation \eqref{eqn:r_bellman_eqn_const_gain0} has a solution.

    \item Section \ref{sec:HD_S} shows that if both the controller and the adversary are communicating and compact (not necessarily convex), and $\mrm{K} = \KS$, then a stationary policy is optimal for the controller if and only if $\alpha' = \alpha^*$. Here, $\alpha'$ denotes the solution to \eqref{eqn:inf_sup_eqn_const_gain}, obtained by swapping the sup-inf order in \eqref{eqn:r_bellman_eqn_const_gain0}.
\end{itemize}

Beyond the technical contributions, we also connect our theory to operations research applications in Section \ref{section:OR_example}. In particular, we discuss robust dynamic service allocation in queuing environments and inventory control, two settings in which long-run stability behavior is a central operational concern. These examples illustrate that unoptimized policies can induce substantially different recurrent classes, so existing average-reward robust MDP results relying on stronger recurrence assumptions may not apply. Our framework accommodates these operational phenomena through more natural and easily verifiable structural conditions, thereby providing a foundation for robust control applications in many operations environments.

\subsection{Literature review}
\textbf{Robust MDPs:} While the Bellman optimality of discounted-reward robust MDPs has been extensively studied \citep{iyengar2005robust,wiesemann2013robust,xu2010distributionally,wang2023foundationRMDP,grand-clement2024tractable}, the corresponding results for the average-reward setting remain underexplored. To the best of our knowledge, \citet{wang2023avg_unichain_dp} provides the first results under strong assumptions of SA-rectangularity and uniform unichains. \citet{grand2023beyond} focus on Blackwell optimality, showing that $\epsilon$-Blackwell optimal policies always exist under SA-rectangularity. Moreover, they demonstrate that in S-rectangular RMDPs, average-reward optimal policies may fail to exist; and even when they do exist, they may need to be strictly history-dependent.

\textbf{Stochastic games (SGs):} S-rectangular robust MDPs can be viewed as a generalization of two-player zero-sum SGs. They extend the standard SGs framework in the following ways: (i) an asymmetry of information, where the controller may use history-dependent policies while the adversary is restricted to stationary or Markovian ones, and (ii) ambiguity-set constraints, where the adversary’s feasible set may be infinite and nonconvex, in contrast to the convex mixed-strategy sets of SGs.   Below, we provide a detailed review of what is known in the stochastic game literature. The properties of Blackwell $\varepsilon$-optimal strategies are subsequently studied in \citet{grand2025playing}.

\citet{iwase1976markov} show that stochastic games have a value and that both players have optimal stationary policies when the Bellman equation admits a solution, which establishes a result similar to our Theorem 1. \citet{mertens1981stochastic} establishes that $\epsilon$-optimal strategies always exist for all players, implying that every zero-sum stochastic game admits a value in the finite-state and finite-action setting. In the special case of irreducible stochastic games, Section~5 of \citet{filar2012competitive} shows that both players possess optimal stationary strategies, that the optimal value is state-independent, and that the Bellman equation admits a solution equal to this common value. A concise overview of these foundational results can also be found in the tutorial by \citet{renault2019tutorial}. To further weaken these assumptions, \citet{wei2017online} shows that when the SG is player~1  communicating, the Bellman equation still admits a solution that characterizes the value of the game. The player~1 communication condition is the SG analogue of the controller-communication assumption in our paper. However, as in the literature on stochastic games, their players' policies are randomized, making the decision sets convex.  
For continuous-state or continuous-action settings, analogous results have been derived under the geometric ergodicity assumption \citep{jaskiewicz2006zero, jaskiewicz2009zero, jaskiewicz2010continuous}. 

Finally, \citet{garrec2019communicating} study communicating zero-sum product stochastic games, where each player has an individual state that evolves solely according to their own previous state and action. However, their notion of communication differs from the one considered in our formulation. In fact, their setting does not satisfy the (weakly) communicating assumption used in this paper. As shown by \citet{garrec2019communicating} as well as \citet{vigeral2013zero, sorin2015reversibility, ziliotto2016zero}, the value of average-reward stochastic games may fail to exist without compactness or communicating assumptions when the action spaces are infinite. Recently, \citet{gaubert2025thresholds} study Blackwell optimality in stochastic games.




\subsection{Comments on Paper Organization}
In the sections that follow, we set the stage and present the paper’s main results. To guide the reader, we provide a brief roadmap and highlight the technical flow.

Section \ref{sec:canonical_construction} gives a rigorous, self-contained formulation of the controller–adversary dynamics and the optimal control objective in a robust MDP. Section \ref{sec:bellman_opt} introduces the constant gain robust Bellman equation and, conditional on the existence of a solution, derives its implications for the optimal robust control problem. Section \ref{sec:exist_sol} provides sufficient conditions for the existence of solutions to the robust Bellman equation, focusing on one-sided (weak) communication-type structures that arise naturally in classical MDPs. Section \ref{sec:HD_S} analyzes a special case where information asymmetry forces the optimal robust control value to equal the value of the robust Bellman equation with the $\sup$ and $\inf$ interchanged; in this regime, no stationary policy can be near-optimal unless the original and exchanged equations have the same constant gain. Finally, Section~\ref{section:OR_example} shows how the proposed framework applies to robust dynamic service allocation and inventory control. These are two canonical operations settings where policy-dependent recurrent behavior makes existing average-reward robust MDP theory insufficient, while our framework applies naturally.

The main theorems are organized as follows. Theorems \ref{thm:general_dpp}, \ref{thm:interchange_imply_dpp}, and \ref{thm:policy_opt_gap} establish consequences \textit{assuming} that the constant gain robust Bellman equation admits a solution. Theorems \ref{thm:existence_of_sol}, \ref{thm:wc_ctrl}, \ref{thm:wc_adv}, and \ref{thm:suff_cond_for_swap}, together with their corollaries, give sufficient conditions on the controller’s and adversary’s decision sets that \textit{guarantee the existence} of a solution. Finally, Theorem \ref{thm:value_HD_S} treats an asymmetric-information setting in which, even when a solution exists, the optimal robust control value need not coincide with that solution.

\section{Canonical Construction and the Optimal Robust Control Problem}\label{sec:canonical_construction}

In this section, we first present a brief but self-contained canonical construction of the probability space, the processes of interest, and the controller’s and adversary’s policy classes. The construction closely follows \citet{wang2023foundationRMDP}, to which we refer the reader for additional details. 

Let $S,A$ be finite state and action spaces, each equipped with the discrete Borel $\sigma$-fields $\cS$ and $\cA$, respectively. Define the underlying measurable space $(\Omega, \cF)$ with $\Omega = (S\times A)^{\Z_{\geq 0}}$ and $\cF$ the corresponding cylinder $\sigma$-field. The process $\{(X_t,A_t), t \geq 0\}$ is defined by point evaluation, i.e., $X_t(\omega) = s_t$ and $A_t(\omega) = a_t$ for all $t \geq 0$ and any $\omega = (s_0,a_0,s_1,a_1,\ldots) \in \Omega$.


For each $t\geq 0$, the history space at time $t$ is
\[
\begin{aligned}
    \bd H_t
    &:=(S \times A)^t\times S \\
    &=
    \bigl\{
    (s_0,a_0,\ldots,s_{t-1},a_{t-1},s_t):
    s_k\in S,\ a_k\in A
    \bigr\}.
\end{aligned}
\]
The history process $H_t:\Omega\to \bd H_t$ is defined by $H_t(\omega)
    =
    (X_0(\omega),A_0(\omega),\ldots,A_{t-1}(\omega),X_t(\omega)),$
and we let  $\cH_t := \sigma(H_t).$

Given a prescribed subset $\cQ \subseteq \cP(\cA)$, a controller policy $\pi$ is a sequence of decision rules $\pi = (\pi_0,\pi_1,\pi_2,\ldots)$ where each $\pi_t$ is a measure-valued function $\pi_t:\bd{H}_t \ra \cQ$, represented in conditional distribution form as $\pi_t(a|h_t) \in [0,1]$ with $\sum_{a \in A}\pi_t(a|h_t) = 1$. { In this paper, we highlight two main types of decision sets:  the deterministic policy set $\cQ = \{\delta_a : a \in {A}\}$ and 
    the fully randomized action set $\cQ = \mathcal{P}(\mathcal{A})$.} The history-dependent controller policy class is therefore $$\PiH(\cQ) := \{ \pi = (\pi_0,\pi_1,\ldots) : \pi_t \in \{\bd{H}_t \ra \cQ\}, \ \forall t \geq 0 \}.$$

A controller policy $\pi = (\pi_0,\pi_1,\ldots)$ is stationary if for any $t_1,t_2 \geq 0$ and $h_{t_1} \in \bd{H}_{t_1}, h'_{t_2} \in \bd{H}_{t_2}$ such that $s_{t_1} = s'_{t_2}$, we have $\pi_{t_1}(\cdot|h_{t_1}) = \pi_{t_2}(\cdot|h'_{t_2})$. In particular, this means $\pi_t(a|h_{t}) = \Delta(a|s_t)$ where $h_t =(s_0,a_0,\ds, s_t) $ for some $\Delta:S \to \cQ$ for all $t \geq 0$. Thus, a stationary controller policy can be identified with $\Delta:S \to \cQ$, i.e., $\pi = (\Delta,\Delta,\ldots)$. Accordingly, the stationary policy class for the controller is $$\PiS(\cQ) := \{ (\Delta,\Delta,\ldots) : \Delta \in \{S \to \cQ\} \},$$ which is identified with $\{S \to \cQ\}$.

On the adversary side, for each $s \in S$ we fix a prescribed set of measure-valued functions $\cP_s \subseteq \{A \to \cP(\cS)\}$. The product set $\cP := \bigtimes_{s \in S} \cP_s$ is called an S-rectangular ambiguity set.

Given $\{\cP_s:s \in S\}$, a history-dependent $S$-rectangular adversary policy $\kappa$ is a sequence of adversarial decision rules $\kappa = (\kappa_0,\kappa_1,\kappa_2,\ldots)$. Each decision rule $\kappa_t$ specifies the conditional distribution of the next state given a history $h_t\in\bd{H}_t$ and an action $a\in A$, i.e., $\kappa_t(s'|h_t,a) \in [0,1]$ with $\sum_{s' \in S}\kappa_t(s'|h_t,a) = 1$. The history-dependent adversary policy class is $$\KH(\cP) := \{ \kappa = (\kappa_0,\kappa_1,\ldots) : \kappa_t(\cdot|h_t,\cdot) \in \cP_{s_t}, \ \text{where } h_t = (s_0,a_0,\ldots,s_t), \ \forall t \geq 0 \}.$$

Analogous to the controller side, a stationary adversary policy $\kappa = (\kappa_0,\kappa_1,\ldots)$ can be identified with $p \in \cP$, i.e., $\kappa = (p,p,\ldots)$ with $\kappa_t(s'|h_t,a) = p(s'|s_t,a)$ where $h_t = (s_0,a_0,\ldots,s_t)$. Thus, the stationary adversary policy class is $\KS(\cP) := \{ (p,p,\ldots): p \in \cP \}$, which can be identified directly with $\cP$.

As shown in \citet{wang2023foundationRMDP}, for $\Pi = \PiH(\cQ)$ or $\PiS(\cQ)$ and $\mrm K = \KH(\cP)$ or $\KS(\cP)$ the triple $\mu\in\cP(\cS),\pi\in\Pi,\kappa\in\mrm K$ uniquely defines a probability measure $P_\mu^{\pi,\kappa}$ on $(\Omega,\cF)$. The expectation under $P_\mu^{\pi,\kappa}$ is denoted by $E_\mu^{\pi,\kappa}$.


This paper considers the optimal robust control of the upper and lower long-run average rewards associated with a robust MDP instance $(\cQ,\cP,r)$ defined by

\[
\begin{aligned}
\overline{\alpha}(\mu,\Pi,\mrm{K})&:= \sup_{\pi\in\Pi}\inf_{\kappa\in\mrm{K}}\overline{\alpha}(\mu,\pi,\kappa),\\
\underline{\alpha}(\mu,\Pi,\mrm{K})&:= \sup_{\pi\in\Pi}\inf_{\kappa\in\mrm{K}}\underline{\alpha}(\mu,\pi,\kappa),
\end{aligned}
\]
where
$$
\begin{aligned}\overline{\alpha}(\mu,\pi,\kappa)&:= \limsup_{n\ra\infty}E_\mu^{\pi,\kappa}\frac{1}{n}\sum_{k=0}^{n-1} r(X_k,A_k),\\
\underline{\alpha}(\mu,\pi,\kappa)&:= \liminf_{n\ra\infty}E_\mu^{\pi,\kappa}\frac{1}{n}\sum_{k=0}^{n-1} r(X_k,A_k).
\end{aligned}$$
Without loss of generality, we assume the reward function $r$ is bounded between 0 and 1.

\par The controller’s policy class is either $\PiH(\cQ)$ or $\PiS(\cQ)$, while the adversary’s policy class is either $\KH(\cP)$ or $\KS(\cP)$. For notational simplicity, we will suppress the dependence of $\Pi$ and $\mrm K$ on $\cQ$ and $\cP$ whenever it is clear from the context.

\section{Robust Bellman Equations and Optimality}\label{sec:bellman_opt}
In this section, we define the constant-gain robust Bellman equation and show that its solution determines the long-run average reward of the robust control problem. This also implies stationary optimality for the controller in the $\overline\alpha(\mu,\PiH,\KH)$ and $\underline\alpha(\mu,\PiH,\KH)$ case. The proofs of the results in this Section are deferred to Appendix \ref{sec:proof_bellman_eqn_opt}. 

\subsection{Robust Bellman Equation with a Constant Gain}
\begin{definition}
    $(u^*,\alpha^*)\in \set{S\ra \R}\times [0,1]$ is said to be a solution of the robust Bellman equation with a constant gain if for any $s\in S$,
    \begin{equation}\label{eqn:r_bellman_eqn_const_gain}
        u^*(s) = \sup_{\phi\in \cQ}\inf_{p_s\in\cP_s}E_{\phi,p_s}[r(s,A_0) - \alpha^* + u^*(X_1)].
    \end{equation}
    Here, the expectation is taken w.r.t. the measure $P_{\phi,p_s}(A_0 = a,X_1 = s') = \phi(a)p_{s,a}(s')$. We say that $(u',\alpha')\in \set{S\ra \R}\times [0,1]$ is a solution to the inf-sup equation with a constant gain if  for any $s\in S$,
    \begin{equation}\label{eqn:inf_sup_eqn_const_gain}
        u'(s) = \inf_{p_s\in\cP_s}\sup_{\phi\in \cQ} E_{\phi,p_s}[r(s,A_0) - \alpha' + u'(X_1)].
    \end{equation}
\end{definition}
We note that if $(u,\alpha)$ solves either \eqref{eqn:r_bellman_eqn_const_gain} or \eqref{eqn:inf_sup_eqn_const_gain}, then $(u+c,\alpha)$ also solves the corresponding equation for any $c\in\R$. 

It is useful to introduce the following discounted robust Bellman equations for a discount factor $\gamma \in (0,1)$. These will serve as key theoretical tools in establishing the existence of solutions to the average-reward equation \eqref{eqn:r_bellman_eqn_const_gain}.

\begin{definition}
    We say that $v_\gamma^*:S\ra \R$ solves the $\gamma$-discounted robust Bellman equation if  for any $s\in S$,
    \begin{equation}\label{eqn:discounted_r_bellman_eqn}
        v_\gamma^*(s) = \sup_{\phi\in \cQ}\inf_{p_s\in\cP_s}E_{\phi,p_s}[r(s,A_0) +\gamma v_\gamma^* (X_1)].
    \end{equation} 
    Similarly, $v'_\gamma :S\ra \R$ solves the $\gamma$-discounted inf-sup equation if  for any $s\in S$,
    \begin{equation}\label{eqn:discounted_inf_sup_eqn}
        v_\gamma'(s) = \inf_{p_s\in\cP_s}\sup_{\phi\in \cQ} E_{\phi,p_s}[r(s,A_0) + \gamma v_\gamma '(X_1)].
    \end{equation}
\end{definition}
\begin{remark}
    In the discounted setting, existence and uniqueness of solutions to \eqref{eqn:discounted_r_bellman_eqn} and \eqref{eqn:discounted_inf_sup_eqn} follow from a standard contraction mapping argument; see \citet{wang2023foundationRMDP}. By contrast, in the average-reward setting, existing results establish the existence of a solution to \eqref{eqn:r_bellman_eqn_const_gain} only in SA-rectangular settings and under additional assumptions \citep[Theorem~8]{wang2023avg_unichain_dp}. We substantially generalize these existence results to one–sided weak communication settings, which are robust analogues of the weakly communicating structures commonly assumed in classical MDPs to ensure constant-gain optimality \citep{puterman2014MDP}.

\end{remark}

\subsection{Bellman Optimality}
\label{subsec:bellmanOpt}

We show that a solution to \eqref{eqn:r_bellman_eqn_const_gain}, if exists, will characterize the optimal robust value. 

\begin{theorem}\label{thm:general_dpp}
    If $(u^*,\alpha^*)$ solves \eqref{eqn:r_bellman_eqn_const_gain}, then 
    \begin{equation}\label{eqn:general_dpp}
    \begin{aligned}
    \alpha^*
    &=\overline\alpha(\mu,\PiH,\KH) = \underline\alpha(\mu,\PiH,\KH) \\
    &=\overline\alpha(\mu,\PiS,\KH) = \underline\alpha(\mu,\PiS,\KH) \\
    & =\overline\alpha(\mu,\PiS,\KS) = \underline\alpha(\mu,\PiS,\KS) \\
    \end{aligned}   
    \end{equation}
    for all $\mu\in\cP(\cS)$. Moreover, any  solution $(u,\alpha)$ to \eqref{eqn:r_bellman_eqn_const_gain} satisfies $\alpha = \alpha^*$. 
\end{theorem}

In particular, the stationary policy class $\PiS$ attains the same optimal value as the fully history-dependent class when playing against a history-dependent adversary whose policies belong to $\KH$; that is,
$\overline{\alpha}(\mu,\PiH,\KH) = \overline{\alpha}(\mu,\PiS,\KH)$ and $\underline{\alpha}(\mu,\PiH,\KH) = \underline{\alpha}(\mu,\PiS,\KH)$. This certifies the optimality of the stationary Markov policies. 

\begin{remark}\label{rmk:left_out_HDS}
Note that the $\PiH$–$\KS$ case is intentionally excluded from Theorem \ref{thm:general_dpp}. This setting leads to a quite curious phenomenon, which we treat separately in Section \ref{sec:HD_S}. Interested readers may wish to proceed directly to that section.
\end{remark}

Next, we show that if a solution to \eqref{eqn:r_bellman_eqn_const_gain} also satisfies \eqref{eqn:inf_sup_eqn_const_gain}, then strong duality holds. In particular, if, in addition, the supremum and infimum are attained, the resulting policy pair constitutes a Nash equilibrium.

\begin{theorem}\label{thm:interchange_imply_dpp}
    If $(u^*,\alpha^*)$ solves \eqref{eqn:r_bellman_eqn_const_gain} and \eqref{eqn:inf_sup_eqn_const_gain}, then $$\sup_{\pi\in\Pi}\inf_{\kappa\in\mrm K}\overline\alpha(\mu,\pi,\kappa) = \inf_{\kappa\in\mrm K}\sup_{\pi\in\Pi}\overline\alpha(\mu,\pi,\kappa) = \sup_{\pi\in\Pi}\inf_{\kappa\in\mrm K}\underline\alpha(\mu,\pi,\kappa) = \inf_{\kappa\in\mrm K}\sup_{\pi\in\Pi}\underline\alpha(\mu,\pi,\kappa) = \alpha^*$$ for every combination of $\Pi = \PiH,\PiS$ and $\mrm K = \KH,\KS$ and any $\mu\in\cP(\cS)$. 
\end{theorem}

\subsection{Optimality of Stationary Policies From the Robust Bellman Equation}
Similar to the classical MDP setting, given $(u^*, \alpha^*)$, any stationary policy that is $\epsilon$-optimal for the robust Bellman equation is also $\epsilon$-optimal for the long-run average reward of the robust MDP, for any $\epsilon \ge 0$.

\begin{theorem}\label{thm:policy_opt_gap}
    Let $(u^*,\alpha^*)$ be a solution to \eqref{eqn:r_bellman_eqn_const_gain}. If for some $\epsilon \geq  0$ and a stationary policy $\Delta:S\ra\cQ$, 
    $$u^*(s) \leq \inf_{p_s\in\cP_s}E_{\Delta(\cd|s),p_s} [r(s,A_0) -\alpha^* + u^*(X_1)]+\epsilon,\quad  \forall s\in S.$$
    Then, $\Delta$ is $\epsilon$-optimal among all stationary polices; i.e. $$\overline{\alpha}(\mu,\PiS,\mrm{K}) - \inf_{\kappa\in\mrm{K}}\underline{\alpha}(\mu,\Delta,\kappa)\leq \epsilon$$
    where $\mrm{K} = \KH$ or $\KS$. 
\end{theorem}
\begin{remark}\label{rmk:policy_from_bellman}
    If we couple Theorem \ref{thm:policy_opt_gap} with Theorem \ref{thm:general_dpp} and \ref{thm:interchange_imply_dpp}, $\Delta$ is also $\epsilon$-optimal for the robust control problem with a history-dependent controller. Moreover, note that if $\epsilon = 0$, i.e. $\Delta(\cd|s)$ achieves the $\sup_{\phi\in\cQ}$ for all $s\in S$, then $\Delta$ is robust optimal. 
\end{remark}

\section{Existence of Solution} \label{sec:exist_sol}
The previous section showed that, when \textit{a solution exists}, the constant-gain robust Bellman equation \eqref{eqn:r_bellman_eqn_const_gain} characterizes the optimal value and a stationary policy. However, there is no reason to expect a constant-gain solution to exist in general, since even multichain Markov reward processes can exhibit state-dependent gain. 

In this section, we develop sufficient conditions--generalizing classical MDP counterparts and motivated by operations research application--that guarantee the existence of a solution. We note that, beyond theoretical interest, the existence theory provides practical guidance for modeling and algorithm design. Our theory equips practitioners with tools to make simple observations or modifications to the model or algorithmic idea to ensure the solvability of the robust Bellman equation.

\subsection{General Criteria}
We begin with a necessary and sufficient condition that links the constant-gain average-reward equation to the discounted version. This result serves as a cornerstone for the subsequent developments.

\begin{theorem}\label{thm:existence_of_sol}
    Given arbitrary $\cQ\subseteq\cP(\cA)$ and $\set{\cP_s\subseteq \set{A\ra \cP(\cS)}:s\in S}$, the following statements are equivalent: 
    \begin{enumerate}[label = {(\arabic*)}]
        \item 
        The solutions $\set{v^*_\gamma:\gamma\in(0,1)}$ to the $\gamma$-discounted equation \eqref{eqn:discounted_r_bellman_eqn} have uniformly bounded span; i.e. $$\sup_{\gamma\in(0,1)}\spnorm{v_\gamma^*}  = \sup_{\gamma\in(0,1)}\sqbk{\max_{s\in S}v_\gamma^*(s)- \min_{s\in S}v^*_\gamma(s)}< \infty.$$ \label{enum:discounted_bd_span}
        \item 
        The constant-gain average-reward robust Bellman equation \eqref{eqn:r_bellman_eqn_const_gain} has a solution $(u^*,\alpha^*)$. \label{enum:arbe_exist_sol}
    \end{enumerate}
\end{theorem}

\subsection{Weakly Communicating Structures}

In this section, we establish that under compactness and the one-sided weakly communicating structures in Definition \ref{def:wc}, both the robust Bellman equation \eqref{eqn:r_bellman_eqn_const_gain} and the inf–sup equation \eqref{eqn:inf_sup_eqn_const_gain} admit solutions. These one-sided weak-communication assumptions are well motivated: they are extensively used in the classical MDP literature to guarantee a constant optimal gain \citep{puterman2014MDP}. In particular, as reviewed earlier, one-sided communication ensures the existence of solutions in SG settings where both players may employ fully randomized strategies \citep{wei2017online}, which in turn induces convex controller policy and adversary ambiguity sets. However, to the best of our knowledge, extensions to the non-convex or the weak-communication setting have not been established in the literature.

Moreover, the robust MDP landscape is more intricate, as applications sometimes necessitate non-convex decision sets for one or both players. We show that, under one-sided weak communication and compactness—without requiring convexity—the constant-gain robust Bellman equation \eqref{eqn:r_bellman_eqn_const_gain} or its inf–sup counterpart \eqref{eqn:inf_sup_eqn_const_gain} admits a solution, with the relevant equation determined by which side satisfies the weak-communication assumption.

We proceed by introducing the following notation. For $p\in \cP$ and $\Delta:S\ra\cP(\cA)$, denote
$$p_\Delta(s'|s):= \sum_{a\in A} p(s'|s,a)\Delta(a|s).$$ Also, let $p_{\Delta}^n(s'|s)$ be the $(s,s')$ entry of the n'th power of the matrix $\set{p_\Delta(s'|s):s,s'\in S}$. Moreover, for $C\subseteq S$ we denote the complement of $C$ is $S$ by $C^c:= S\setminus C$. 

\begin{definition}[Weak Communication]\label{def:wc}
Consider arbitrary controller and adversary action sets $\cQ\subseteq\cP(\cA)$ and $\cP=\bigtimes_{s\in S}\cP_s$, with $\cP_s\subseteq\set{A\ra\cP(\cS)}$.

\begin{itemize}
\item A stationary controller policy $\Delta:S\ra\cQ$ is said to be weakly communicating if there is a communicating class $C_\Delta\subseteq S$ s.t. for any $s,s'\in C_\Delta$, there exists $p\in\cP$ and $N\geq 1$ s.t. $p_\Delta^N(s'|s) > 0$. Moreover, for all $s\in C_\Delta^c$, $s$ is transient under any stationary adversarial policy. 
     
The controller is weakly communicating if every stationary policy $\Delta:S\ra\cQ$ is weakly communicating. 
     
\item A stationary adversary policy $p\in\cP$ is said to be weakly communicating if there is a communicating class $C_p\subseteq S$ s.t. for any $s,s'\in C_p$, there exists $\Delta:S\ra\cQ$ and $N\geq 1$ s.t. $p_\Delta^N(s'|s) > 0$. Moreover, for all $s\in C_p^c$, $s$ is transient under any stationary controller policy. 
     
The adversary is weakly communicating if every stationary policy $p \in \cP$ is weakly communicating. 
\end{itemize}
\end{definition}

\begin{remark}Note that our weak communication definitions parallel their classical MDP counterpart. Moreover, a controller/adversary may be weakly communicating even when the communicating sets $C_\Delta$/$C_p$ depend on the particular stationary policy: they need not coincide across stationary policies. 
\end{remark}

Communicating controller and adversary are defined analogously as follows.

\begin{definition}[Communication]\label{def:comm} Consider arbitrary controller and adversary action sets $\cQ\subseteq\cP(\cA)$ and $\cP=\bigtimes_{s\in S}\cP_s$, with $\cP_s\subseteq\set{A\ra\cP(\cS)}$.
    \begin{itemize}
    \item A stationary controller policy $\Delta:S\ra\cQ$ is communicating if it is weakly communicating with $C_\Delta = S$. The controller is communicating if every $\Delta:S\ra\cQ$ is communicating. 

    \item A stationary adversary policy $p\in\cP$ is communicating if it is weakly communicating with $C_p = S$. The adversary is communicating if every $p\in\cP$ is communicating. 
\end{itemize}
\end{definition}

With these definitions, we are ready to state the main results of this section. 

\begin{theorem}[Controller-Side Structures]\label{thm:wc_ctrl}
Assume either of the two assumptions holds: 
\begin{enumerate}[label = {\textrm(\arabic*)}]
    \item The controller is weakly communicating and $\cQ$ and $\cP_s$ are compact for all $s\in S$. 
    \item The controller is communicating and $\cQ$ is compact.
\end{enumerate}
Then the constant gain average reward robust Bellman equation \eqref{eqn:r_bellman_eqn_const_gain} has a solution.
\end{theorem}
The proof of Theorem \ref{thm:wc_ctrl} is provided within the main body of the paper in Section \ref{sec:proof_of_thm_wc_ctrl}. 

Symmetric to the controller-side results, we show that the adversary-side weak communication structures will imply the existence of solutions to \eqref{eqn:inf_sup_eqn_const_gain}. The proof for Theorem \ref{thm:wc_adv} is deferred to Appendix \ref{sec:proof_wc_adv} as it is similar to the proof of Theorem \ref{sec:proof_of_thm_wc_ctrl}.

\begin{theorem}[Adversary-Side Structures]\label{thm:wc_adv}
Assume either of the two assumptions holds: 
\begin{enumerate}[label = {\textrm(\arabic*)}]
    \item The adversary is weakly communicating and $\cQ$ and $\cP_s$ are compact for all $s\in S$. 
    \item The adversary is communicating and $\cP_s$ is compact for all $s\in S$.
\end{enumerate}
Then the constant gain average reward inf-sup equation \eqref{eqn:inf_sup_eqn_const_gain} has a solution.
\end{theorem}

Note that Theorem \ref{thm:wc_adv} does not, by itself, guarantee a solution to \eqref{eqn:r_bellman_eqn_const_gain}. Nevertheless, under suitable sufficient conditions, the solution $(u',\alpha')$ of \eqref{eqn:inf_sup_eqn_const_gain} also satisfies \eqref{eqn:r_bellman_eqn_const_gain}. We state these conditions formally in Theorem \ref{thm:suff_cond_for_swap}, where we let $\delta_a \in \mathcal{P}(\mathcal{A})$ denote the Dirac measure at $a \in \mathcal{A}$. The proof of Theorem \ref{thm:suff_cond_for_swap} is provided in Appendix \ref{sec:proof:thm:suff_cond_for_swap}.

\begin{theorem}\label{thm:suff_cond_for_swap}
Either of the following conditions is sufficient for a solution $(u',\alpha')$ of \eqref{eqn:inf_sup_eqn_const_gain} to also solve \eqref{eqn:r_bellman_eqn_const_gain}:
\begin{enumerate}[label={\textrm(\arabic*)}]
    \item $\cQ$ and each $\cP_s$ are convex for all $s\in S$, and either $\cQ$ is compact or all $\cP_s,s\in S$ are compact.
    \item $\set{\delta_a : a \in A} \subseteq \cQ$, i.e., the controller can take deterministic actions, and for every $s \in S$, the ambiguity set factorizes as $\cP_s = \bigtimes_{a \in A} \cP_{s,a}$ for some $\cP_{s,a} \subseteq \cP(S)$. In this case, we refer to $\cP$ as an SA-rectangular adversarial ambiguity set.
\end{enumerate}
\end{theorem}

To conclude this section, we note that while Theorems~\ref{thm:wc_ctrl}--\ref{thm:suff_cond_for_swap} provide conditions guaranteeing the existence of a solution and Theorem~\ref{thm:general_dpp} implies the uniqueness of $\alpha^*$, the function $u$ solving \eqref{eqn:r_bellman_eqn_const_gain} or \eqref{eqn:inf_sup_eqn_const_gain} need not be unique up to a constant under weak communication. In particular, even when $\cP=\set{p}$ is a singleton, the MDP literature has shown that multiple functions $u$ may satisfy the non-robust Bellman equation \citep{puterman2014MDP}, which is a special case of \eqref{eqn:r_bellman_eqn_const_gain} and \eqref{eqn:inf_sup_eqn_const_gain}.

\subsection{Useful Corollaries}
Although the conditions of Theorems \ref{thm:wc_ctrl} and \ref{thm:wc_adv} are self-explanatory, and their verification mirrors that of their non-robust counterparts, in this section we highlight several scenarios, motivated by operations applications, in which Theorems \ref{thm:wc_ctrl} and \ref{thm:wc_adv} apply.

Stability structures are ubiquitous in operations research applications: reasonable policies typically lead to systems that are stable and insensitive to their initial states. We show that, under stability assumptions such as (weak) irreducibility or (strong) unichain, Theorems \ref{thm:wc_ctrl} and \ref{thm:wc_adv} are verified.

\begin{corollary}[Irreducible]\label{cor:exists_irreducible} 1. Assume that $\cQ$ is compact and that for each controller's stationary policy $\Delta:S\ra\cQ$ there exists $p\in\cP$ such that $p_{\Delta}$ is irreducible. Then the constant gain average reward robust Bellman equation \eqref{eqn:r_bellman_eqn_const_gain} has a solution. 

2. Assume that $\cP_s$, $s\in S$ are compact and for each $p\in \cP$ there exists $\Delta:S\ra\cQ$ such that $p_{\Delta}$ is irreducible. Then the constant gain inf-sup equation \eqref{eqn:inf_sup_eqn_const_gain} has a solution. 
\end{corollary}
We note that Corollary \ref{cor:exists_irreducible}  allows the choice of $p$ (or $\Delta$) to depend on the stationary policy $\Delta$ (or $p$). Therefore, this is a relatively weak irreducibility condition for the robust MDP. In contrast, when we pass from irreducibility to the unichain assumption in the following Corollary \ref{cor:all_unichain}, we strengthen the requirement by assuming that $p_\Delta$ is unichain for all $p\in\cP$ and all $\Delta:S\to\cQ$.
\begin{remark}
    Corollary \ref{cor:exists_irreducible} could be easy to verify in applications where the ambiguity set is a distributional ball around a nominal kernel $p_0\in\cP$. In particular, if either (i) $p_0$ induces an irreducible Markov chain for every policy, or (ii) $\cP$ contains transition probabilities with full support (e.g., any S- and SA-rectangular total variation \citep{yang2022toward}, $L^p$ \citep{clavier2024near}, and Wasserstein \citep{clement2021WD_DRMDP} ambiguity sets with a radius $\delta > 0$), then \eqref{eqn:r_bellman_eqn_const_gain} admits a solution.
\end{remark}

\begin{proof}[Proof of Corollary \ref{cor:exists_irreducible}]
By the assumptions in the first claim, for any $\Delta:S\ra\cQ$, there is $p\in \cP$ so that $p_{\Delta}$ is irreducible. In particular, all states communicate under $p_{\Delta}$; i.e. $\forall s,s'\in S$, $p_\Delta^N(s'|s) > 0$ for some $N\geq 1$. This verifies the assumption (2) of Theorem \ref{thm:wc_ctrl} and implies the first statement of Corollary \ref{cor:exists_irreducible}. 

For the second statement, the same argument shows that (2) in Theorem \ref{thm:wc_adv} is satisfied.
\end{proof}

Next, we consider the unichain case, in which a closed recurrent class may coexist with additional transient states.

\begin{definition}[Unichain]
    A transition kernel $Q:S\ra\cP(\cS)$ is unichain if $Q$ has only one closed recurrent class. A controlled transition kernel $p:S\times A\ra\cP(\cS)$ is unichain under the stationary controller policy class $S\ra\cQ$ if for all $\Delta :S\ra\cQ$, $p_{\Delta}$ is unichain. 
\end{definition}
\begin{corollary}[Unichain]\label{cor:all_unichain}
    Assume that $\cQ$ and $\cP_s:s\in S$ are compact. If all $p\in\cP$ is unichain, then the constant gain average reward robust Bellman equation \eqref{eqn:r_bellman_eqn_const_gain} has a solution. 
\end{corollary}
\begin{proof}[Proof of Corollary \ref{cor:all_unichain}]
Fix $\Delta:S\ra\cQ$. For every $p\in\cP$, let $R_{\Delta}(p)\subseteq S$ denote the closed recurrent class of $p_\Delta$. We define 
$$C_{\Delta}:=\bigcup_{p\in\cP} R_\Delta(p)$$ and show that $\Delta$ is weakly communicating with communicating class $C_{\Delta}$.

Consider fixed $(s,s')\in C_{\Delta}$. Then, by construction, there exists $q\in \cP$ s.t. $s'\in R_{\Delta}(q)$. Since $s'$ is recurrent under $q_{\Delta}$, $s$ can reach $s'$; i.e. $q^N_\Delta(s'|s) > 0$ for some $N\geq 1$.

On the other hand, for $x\in C_{\Delta}^c$, $x\notin R_\Delta(p)$ for any $p\in\cP$; i.e. $x$ is transient for all $p\in \cP$. 

Therefore, $\Delta$ is weakly communicating. As $\Delta:S\ra\cQ$ is arbitrary, we conclude that the controller is weakly communicating. Thus, the assumption (1) of Theorem \ref{thm:wc_ctrl} is satisfied, implying Corollary \ref{cor:all_unichain}. 
\end{proof}

Corollary \ref{cor:all_unichain} extends the dynamic programming results of \citet{wang_model-free_2023} to S-rectangular and non-convex settings. Nonetheless, the assumption that every $p\in\cP$ is unichain remains strong. For example, Corollary \ref{cor:all_unichain} does not apply when certain stationary policy pairs induce a kernel $p_\Delta$ with multiple recurrent classes.

Building on Theorems \ref{thm:wc_ctrl} and \ref{thm:wc_adv}, we establish a more general result, Corollary \ref{cor:overlapped_recurrent_class}, which allows stationary policies to induce chains with multiple recurrent classes. Before we state the result, for fixed $\Delta:S\ra\cQ$ and $p\in\cP$, we define the collections of recurrent classes $$\begin{aligned}
\cR_{\Delta}&:=\set{R\subseteq S: \exists p\in\cP \text{ such that $R$ is a closed communicating class of $p_\Delta$}}, \\
\cR_{p}&:=\set{R\subseteq S: \exists \Delta:S\ra\cQ \text{ such that $R$ is a closed communicating class of $p_\Delta$}}. 
\end{aligned}$$
Here, a closed communicating class is a set of states that all communicate with each other and from which no state can reach any state outside the set \citep{norris1998markov}.
\begin{definition}[Overlap-Connected Closed Communicating Classes (OCCCC)]\label{def:OCCCC} We say that $\cR_{\Delta}$ (or $\cR_p$) is overlap-connected if for each $R,R'\in\cR_{\Delta}$ (or $\cR_p$), there exists integer $k\geq 1$ and $R_0,R_1,\ds,R_k\in \cR_{\Delta}$ (or $\cR_p$) such that $R_0 = R$, $R_k = R'$, and $R_i\cap R_{i+1}\neq \varnothing$ for all $0\leq i\leq k-1$. 
\end{definition}

In other words, an overlap-connected family of sets $\mathcal R_{\Delta}$ means that for any two sets in the family, they can be reached from one to the other by a finite chain of sets in $\mathcal R_{\Delta}$ such that each consecutive pair in the chain has a nonempty intersection. The same interpretation also applies to
$\mathcal R_{p}$.

As a mnemonic, the acronym \textsc{OCCCC} visually evokes the \mbox{$\mathrm{HO\!-\!CH_2\!-\!CH_2\!-\!CH_2\!-\!CH_3}$} molecular structure of 1-butanol.

\begin{corollary}[OCCCC]\label{cor:overlapped_recurrent_class}
Assume that $\cQ$ is compact and all $\cP_s$, $s\in S$ are convex and compact. If $\cR_{\Delta}$ is overlap-connected for all $\Delta:S\ra\cQ$, then \eqref{eqn:r_bellman_eqn_const_gain} has a solution. 

Assume that  $\cP_s$, $s\in S$ are compact and $\cQ$ is convex and compact. If $\cR_{p}$ is overlap-connected for all $p\in\cP$, then \eqref{eqn:inf_sup_eqn_const_gain} has a solution. 
\end{corollary}

The proof of Corollary \ref{cor:overlapped_recurrent_class} is deferred to Appendix \ref{sec:proof_cor_overlapped_recurrent_class}.

\section{Proof of Theorem \ref{thm:wc_ctrl}}\label{sec:proof_of_thm_wc_ctrl}

In this section, we present the main argument for Theorem \ref{thm:wc_ctrl}, establishing the existence of a solution to the robust Bellman equation \eqref{eqn:r_bellman_eqn_const_gain}. Some intermediate lemmas are deferred to the appendix for clarity. We note that the proof of Theorem \ref{thm:wc_adv} follows a similar strategy, with the roles of the controller and adversary interchanged (see Appendix \ref{sec:proof_wc_adv}).

Our proof primarily addresses the weakly communicating case in assumption (1). Since a communicating controller is a special case, the argument under assumption (2) carries over with only minor changes. Importantly, assumption (2) does not require compactness of $\cP_s$ for $s\in S$. We will explain how the proof extends under this weaker assumption and why compactness is unnecessary in that case; see Remark \ref{rmk:cmpct_P_not_needed} and \ref{rmk:cmpct_P_not_needed_2}.

By Theorem \ref{thm:existence_of_sol}, it suffices to show that under the assumptions of Theorem \ref{thm:wc_ctrl}, the solution $v^*_\gamma$ to \eqref{eqn:discounted_r_bellman_eqn} has uniformly bounded span. To show this, we take the following steps.

\subsection*{Some Preliminary Constructions}

We begin by stating and postponing the proof of a lemma that captures a straightforward implication of a controller policy $\Delta$ being weakly communicating.

\begin{lemma}\label{lemma:wc_ctrl_implication}
If the stationary controller policy $\Delta$ is weakly communicating and $\cP$ is S-rectangular, then for any $w\in S$ and any $y\in C_\Delta$, there exist $p\in\cP$ and $N\le |S|$ such that $p_\Delta^{N}(y|w)>0$.
\end{lemma}
The proof of this lemma is given in Section \ref{sec:proof_of_aux_lemma_wc_ctrl}.

Next, we will leverage the compactness of $\cQ$ to construct a finite subset $B$ of stationary controller policies that yields a uniform lower bound on hitting probabilities. We then use $B$ to obtain a version of Lemma \ref{lemma:wc_ctrl_implication} that is uniform over all stationary controllers; see Lemma \ref{lemma:wc_ctrl_exists_adv_lower_bd_prob}.

Assume the controller is weakly communicating, $\cP$ is S-rectangular, and $\cQ$ is compact. By Lemma \ref{lemma:wc_ctrl_implication}, for any stationary policy $\Delta:S\ra\cQ$ and any $w\in S$, $y\in C_\Delta$, there exist $p\in\cP$ and $N \le |S|$ (both possibly depending on $(w,y,\Delta)$) such that $p_\Delta^{N}(y|w)>0.$

Note that if we fix $p = p^{w,y,\Delta}$ and $N = N^{w,y,\Delta}$ given by Lemma \ref{lemma:wc_ctrl_implication} for this particular $(w,y,\Delta)$, then the mapping
$\eta\ra p^N_\eta$  is continuous for $\eta:S\ra \cP(\cA)$. This is because $\eta \ra p_\eta$,  $p_{\eta}\ra p_{\eta}^N$ and $p_\eta\ra M_{\eta,p}^\Delta$ are continuous. So, there must exists an open neighborhood $O_\Delta^{w,y}\subseteq\set{S\ra\cP(\cA)}$ of $\Delta$ s.t.\begin{equation}
    p_{\eta}^{N} (y|w) \geq p_{\Delta}^{N}(y|w)/2,\quad\forall \eta\in \overline{O^{w,y}_\Delta}.\label{eqn:to_use_within_O_prob_bound}
\end{equation}
where $\overline{O^{w,y}_\Delta}$ denotes the closure of $O^{w,y}_\Delta$. We define 
$$O_\Delta := \bigcap_{w\in S, y\in C_\Delta}O^{w,y}_\Delta.$$ 
 Due to the finiteness of $S$ and $ C_\Delta$, $O_\Delta$ is still open.  Moreover, $\set{O_\Delta:\Delta\in \set{S\ra\cQ} }$ forms an open cover of $\set{S\ra\cQ}$.

We also consider another (not necessarily open) cover $\set{K_\Delta:\Delta\in \set{S\ra\cQ} }$ of $\set{S\ra\cQ}$. For any fixed $\Delta\in\set{S\ra\cQ}$, if $C_\Delta = S$, i.e. the entire state space is communicating, then let $K_\Delta = \set{S\ra \cQ}$. 

On the other hand, if $C_{\Delta}^c\neq \varnothing$, we construct $K_\Delta$ as in the following Lemma. 
\begin{lemma}\label{lemma:construct_K_Delta}
    Assume the controller is weakly communicating, $\cP$ is S-rectangular, and $\cQ$ and $\cP$ are compact. Then, for each $\Delta:S\ra\cQ$ with $C_{\Delta}^c\neq \varnothing$, there exists an open neighborhood $K_\Delta$ of $\Delta$ such that
    \begin{equation}
    \begin{aligned}
    \label{eqn:to_use_within_K_eval_bd}  0& \leq \sup_{p\in\cP}e^\top (I-M_{\eta,p}^\Delta) \inv  e  \\
    &\leq \sup_{p\in\cP}e^\top (I-M_{\Delta,p}^\Delta)\inv e + 1 < \infty,\quad \forall \eta\in \overline{K_\Delta},
    \end{aligned}
\end{equation}
where both suprema are attained. In this expression, $M_{\eta,p}^\Delta$ is the principal submatrix of $p_{\eta}$ on $C^c_{\Delta}$ defined by $M_{\eta,p}^\Delta(s,s') = p_{\eta}(s'|s)$ for $s,s'\in C^c_{\Delta}$, and $e$ denotes the all-ones vector in $\R^{|C_\Delta^c|}$.
\end{lemma}

Therefore, we have defined $O_\Delta$ and $K_\Delta$ for all $\Delta:S\ra \cQ$. With these constructions, we define $$G_\Delta:= O_\Delta\cap K_{\Delta}. $$
Note that when $C_\Delta^c = \varnothing$, then $G_\Delta = O_\Delta\ni \Delta$ is non-empty and open. When $C_\Delta^c \neq \varnothing$, both $O_\Delta$ and $K_\Delta$ are open neighborhood of $\Delta$. Therefore, $\set{G_\Delta:\Delta\in\set{S\ra\cQ}}$ is an open cover of $\set{S\ra\cQ}$.

Since $\cQ$ is compact, the controller policy set $\set{S\ra \cQ}$, seen as stochastic matrices in $\R^{|S|\times |A|}$, is also compact.  Hence, there exists a finite sub-cover $\set{G_{\Delta}:\Delta\in B}$ where $B :=\set{\Delta_1,\ds,\Delta_{|B|}}\subseteq \set{S\ra \cQ}$ is a finite subset.

With this construction, we prove the following Lemma. 
\begin{lemma}\label{lemma:wc_ctrl_exists_adv_lower_bd_prob}
Under the assumptions of Theorem \ref{thm:wc_ctrl}, there exists $\delta>0$ such that, for any stationary controller policy $\Delta:S\to\cQ$ with $\Delta\in G_{\Delta_k}$ and any $y\in C_{\Delta_k}, w\in S$, there exists $p\in\cP$ and $N \leq |S|$ such that $p_\Delta^{N}(y|w)\geq \delta$. 
\end{lemma}

\begin{proof}[Proof of Lemma \ref{lemma:wc_ctrl_exists_adv_lower_bd_prob}]

Fix $w\in S$ and $\Delta:S\ra\cQ$. As $\set{G_{\Delta}:\Delta\in B}$ covers $\set{S\ra\cQ}$, for any stationary control policy $\Delta$, there is $k \in\set{1,\ds,|B|}$ s.t. $\Delta\in G_{\Delta_k}$. We consider an arbitrary $y\in C_{\Delta_k}$. 

Since $B$ is finite, by the construction of $G_\Delta$, we have that \begin{equation}\min_{j = 1,\ds,| B|} \min_{y\in C_{\Delta_j}}(p^{w,y,\Delta_j}_{\Delta_j})^{N^{w,y,\Delta_j}} (y|w) =: \delta_w > 0,\label{eqn:to_use_B_policy_prob_bd}
\end{equation}
where $p^{w,y,\Delta_j}\in\cP$ and $N^{w,y,\Delta_j}\leq |S|$ is given by Lemma \ref{lemma:wc_ctrl_implication}. Note that $\delta_w$ is independent of $\Delta_j$ and $y$. Since the state space is finite, we define $\delta:= \min_{w\in S}\delta_w/2 > 0$. 

To show Lemma \ref{lemma:wc_ctrl_exists_adv_lower_bd_prob}, we choose $p = p^{w,y,\Delta_k}$ and $N = N^{w,y,\Delta_k}\leq |S|$, again given by Lemma \ref{lemma:wc_ctrl_implication}. Then
$$p_\Delta^N(y|w)\geq p_{\Delta_k}^{N}(y|w)/2 \geq \delta_w/2 \geq\delta > 0$$
where the first inequality follows from the construction of $G_{\Delta_k}$ in \eqref{eqn:to_use_within_O_prob_bound} and the second inequality is because of \eqref{eqn:to_use_B_policy_prob_bd}. Also note that by Lemma \ref{lemma:wc_ctrl_implication}, $N = N^{w,y,\Delta_k} \le |S|$.  This proves Lemma \ref{lemma:wc_ctrl_exists_adv_lower_bd_prob}.     
\end{proof}
    
\begin{remark}\label{rmk:cmpct_P_not_needed}
Notice that in the proof of Lemma 3, only the property of $O_\Delta$ was used, not that of $K_\Delta$. Thus, to establish Lemma \ref{lemma:wc_ctrl_exists_adv_lower_bd_prob}, it suffices to construct a finite subcover from the collection $\set{O_\Delta:\Delta\in\set{S\to\cQ}}$. Since the construction of $O_\Delta$ does not require the compactness of $\cP_s$ for $s\in S$, Lemma 3 holds even without assuming compactness of the adversary’s ambiguity set. The same observation applies to Lemma \ref{lemma:wc_bd_E_hitting_time}.
\end{remark}
    
\subsection*{Decomposing $\spnorm{v_\gamma^*}$}

    First, note that $v_\gamma^*$ solves \eqref{eqn:discounted_r_bellman_eqn}, then for each $\epsilon > 0$, there exists $\Delta_\epsilon:S\ra\cQ$ s.t.  for all $s\in S$
    $$\begin{aligned} v^*_\gamma(s) &\leq  \inf_{p_s\in\cP_s}E_{\Delta_\epsilon(\cd|s),p_s}[r(s,A_0) +\gamma v_\gamma^* (X_1)] + (1-\gamma)\epsilon.
    \end{aligned}$$

    Then, by Theorem 1\&5 in \citet{wang2023foundationRMDP}, for $\kappa\in\KH$ denoting $$ v_\gamma^{\Delta_\epsilon,\kappa} (s):=E_s^{\Delta_\epsilon,\kappa}\sum_{k=0}^\infty\gamma^k r(X_k,A_k),$$ we have that for all $s\in S$,
    $$0\leq v_\gamma^*(s) - \inf_{\kappa\in\KS}v_\gamma^{\Delta_\epsilon,\kappa}(s)\leq \epsilon; $$
    i.e. $\Delta_\epsilon$ is $\epsilon$-optimal. Moreover, there exists stationary $p_\epsilon\in \KS$ s.t. for all $s\in S$
    $$0\leq v_\gamma^{\Delta_\epsilon,p_\epsilon}(s) -  \inf_{\kappa\in\KS}v_\gamma^{\Delta_\epsilon,\kappa}(s)\leq \epsilon.$$

    Let $s_\vee,s_\wedge\in S$ so that $v_\gamma^*(s_\vee) = \max_{s\in S}v_\gamma^*(s)$ and $v_\gamma^*(s_\wedge) = \min_{s\in S}v_\gamma^*(s)$. Recall that $\set{G_\Delta:\Delta\in  B}$ constructed in step 2 is an open cover of $\set{S\ra\cQ}$. Then, $\Delta_\epsilon\in G_{\Delta_k}$ for some $k \leq |B|$. On the other hand, by the definition of $\Delta_\epsilon$ and $p_\epsilon$,  \begin{equation}\label{eqn:to_use_vspan_decomp}
       \begin{aligned}
    \spnorm{v_\gamma^*} &= v_{\gamma}^*(s_\vee)-v_\gamma^*(s_\wedge) \\
    &= \sup_{\pi\in\PiS}\inf_{\kappa\in\KH}v_\gamma^{\pi,\kappa}(s_\vee) - \sup_{\pi\in\PiS}\inf_{\kappa\in\KH}v_\gamma^{\pi,\kappa}(s_\wedge)  \\
    &\leq \inf_{\kappa\in \KH}v_{\gamma}^{\Delta_\epsilon,\kappa}(s_\vee)-\inf_{\kappa\in \KH}v_{\gamma}^{\Delta_\epsilon,\kappa}(s_\wedge) + \epsilon\\
    &\leq v_{\gamma}^{\Delta_\epsilon,\kappa}(s_\vee)-v_{\gamma}^{\Delta_\epsilon,p_\epsilon}(s_\wedge) + 2\epsilon.\\
    &= \underbrace{v_{\gamma}^{\Delta_\epsilon,\kappa}(s_\vee)- v_\gamma^{\Delta_\epsilon,p_\epsilon}(y_\wedge) }_{\xi_1} \\
    &+\underbrace{v_\gamma^{\Delta_\epsilon,p_\epsilon}(y_\wedge) - v_{\gamma}^{\Delta_\epsilon,p_\epsilon}(s_\wedge) }_{\xi_2}+ 2\epsilon.
    \end{aligned} 
    \end{equation}
    for any $\kappa\in\KH$, where \begin{equation}\label{eqn:to_use_def_y_wedge}
        y_\wedge \in\argmin_{y\in  C_{\Delta_k}} v_\gamma^{\Delta_\epsilon,p_\epsilon}(y).
    \end{equation}

    Next, we will upper bound $\xi_1$ and $\xi_2$ separately. Before we proceed, we make the following note.

    \begin{remark}\label{rmk:cmpct_P_not_needed_1}
        Note that if the controller is communicating, $C_{\Delta_k} = S$. So, $$\xi_2 = \min_{s\in S}v_{\gamma}^{\Delta_\epsilon,p_\epsilon}(s) - v_{\gamma}^{\Delta_\epsilon,p_\epsilon}(s_\wedge) \leq 0.$$ So, under assumption (2) of the theorem, we only need to uniformly bound $\xi_1$, which will not require the compactness of $\cP_s$ for $s\in S$; see Remark \ref{rmk:cmpct_P_not_needed_2}.  
    \end{remark}

    \subsection*{Upper-Bounding $\xi_1$}
    We will upper bound $\xi_1$ using the expected hitting time of $y$. To proceed, we first state Lemma \ref{lemma:wc_bd_E_hitting_time}, with the proof deferred to Section \ref{sec:proof_of_aux_lemma_wc_ctrl}.
    
    \begin{lemma}\label{lemma:wc_bd_E_hitting_time}
        Under the assumptions of Lemma \ref{lemma:wc_ctrl_exists_adv_lower_bd_prob}, there exists $\delta' > 0$ s.t. for any stationary controller policy $\Delta:S\ra\cQ$ with $\Delta\in G_{\Delta_k}$ and $y\in C_{\Delta_k}$, there exists $q\in\cP$ such that 
        $$\max_{w\in S}E^{\Delta,q}_w \tau_y \leq \frac{|S|}{\delta'}.$$
    \end{lemma}

With Lemma \ref{lemma:wc_bd_E_hitting_time}, we can prove a uniform upper bound for $\xi_1$. For the convenience of the proceeding proofs, rather than focusing on $s_\vee$ and $y_\wedge$, we will assume that $x\in S$ is an arbitrary initial state and that $y\in 
    C_{\Delta_k}$ is an arbitrary communicating state. 
    
    Since $\epsilon > 0$ can be arbitrarily small, it suffices to choose a $\kappa\in \KH$ (potentially depending on $\Delta_\epsilon,p_\epsilon,x,y$) so that $v_{\gamma}^{\Delta_\epsilon,\kappa}(x)-v_{\gamma}^{\Delta_\epsilon,p_\epsilon}(y)$ is uniformly bounded in $\gamma$. To achieve this, we will use a two-phase adversarial policy similar to that in \citet{bartlett2009WC}. We consider a history-dependent adversary $\kappa = (\kappa_0,\kappa_1,\ds)\in\KH$ defined as follows. Let $g_{t-1} = (s_0,a_0,\ds,s_{t-1},a_{t-1})$ and
    \begin{equation}\label{eqn:to_use_def_kappahd}\kappa_t(s'|g_{t-1},s,a) = \begin{cases}
        q(s'|s,a)&\text{if } s_k\neq y, \forall k\leq t-1\text{ and } s\neq y,\\
        p_\epsilon(s'|s,a) & \text{otherwise.}
    \end{cases}
    \end{equation}
    Here, since $\Delta_\epsilon\in G_{\Delta_k}$ and $y\in C_{\Delta_k}$, we choose $q = q^{y,\Delta_\epsilon}$ defined in Lemma \ref{lemma:wc_bd_E_hitting_time}. In other words, the $\kappa$ uses $q$ when the chain hasn't hit $y$ and uses the $\epsilon$-optimal adversary after hitting $y$. 

    Under this history-dependent adversarial policy, we have for all $x\in S$
    \begin{equation}\label{eqn:to_use_vkappa_hd_bd}
    \begin{aligned}
        &v_{\gamma}^{\Delta_\epsilon,\kappa}(x) \\
        =& E^{\Delta_\epsilon,\kappa}_x \sum_{k=0}^\infty \gamma^k r(X_k,A_k)\\
        =& E^{\Delta_\epsilon,\kappa}_x \sum_{k=0}^{\tau_y-1} \gamma^k r(X_k,A_k)  \\
        +& E^{\Delta_\epsilon,\kappa}_x \gamma^{\tau_y}\sum_{k=\tau_y}^{\infty} \gamma^{k-\tau_y} r(X_k,A_k)\\
        \leq& E_x^{\Delta_\epsilon,\kappa}\tau_y + E^{\Delta_\epsilon,\kappa}_x \sum_{k=\tau_y}^{\infty} \gamma^{k-\tau_y} r(X_k,A_k).
    \end{aligned}    
    \end{equation}
    Note that by the construction of $\kappa$ in \eqref{eqn:to_use_def_kappahd}, we have
    \begin{equation} \label{eqn:to_use_E_kappa_tau_y}\begin{aligned}
    E_x^{\Delta_\epsilon,\kappa}\tau_y 
    &{=} \sum_{k=0}^\infty P_x^{\Delta_\epsilon,\kappa}(\tau_y >  k)\\
    &= \sum_{k=0}^\infty E_x^{\Delta_\epsilon,\kappa}\1\set{X_0(\omega),X_1(\omega),\ds,X_k(\omega)\neq y}\\
    &\stackrel{(i)}{=} \sum_{k=0}^\infty \sum_{g_k = (s_0,a_0,\ds,s_k,a_k)\in\bd G_k}\prod_{j=0}^{k-1}\Delta_\epsilon(a_j|s_j)\kappa_j(s_{j+1}|g_{j-1},s_j,a_j)\1\set{s_0,\ds,s_k\neq y}\\\
    &\stackrel{(ii)}{=} \sum_{k=0}^\infty \sum_{g_k\in\bd G_k}\prod_{j=0}^{k-1}\Delta_\epsilon(a_j|s_j)q(s_{j+1}|s_j,a_j) \1\set{s_0,\ds,s_k\neq y}\\
    &= E_x^{\Delta_\epsilon,q}\tau_y .\end{aligned}
    \end{equation}
    Here, $(i)$ follows from the definition of $E_x^{\Delta_\epsilon,\kappa}$, $(ii)$ is because by \eqref{eqn:to_use_def_kappahd}
    $$\kappa_j(s_{j+1}|g_{j-1},s_j,a_j)\1\set{s_0,\ds,s_k\neq y} = \begin{cases}
        0 & \text{if }s_i =  y\text{ for some }i\leq j\\
        q(s_{j+1}|s_j,a_j) &\text{if }s_0,\ds,s_j\neq y 
    \end{cases},$$
    and the last equality follows from reversing the previous steps. 

    On the other hand, we observe $$\begin{aligned}
        &E^{\Delta_\epsilon,\kappa}_x \sum_{k=\tau_y}^{\infty} \gamma^{k-\tau_y} r(X_k,A_k)\\
        &= E^{\Delta_\epsilon,\kappa}_x \sum_{j=1}^\infty\sum_{k=j}^{\infty} \1\set{\tau_y = j}\gamma^{k-\tau_y} r(X_k,A_k)\\
        &\stackrel{(i)}{=} \sum_{j=1}^\infty\sum_{k=j}^{\infty}  E^{\Delta_\epsilon,\kappa}_x  \1\set{\tau_y = j}\gamma^{k-j}r(X_k,A_k)\\
        &= \sum_{j=1}^\infty\sum_{k=j}^{\infty}  \sum_{g_k\in \bd{G}_k}\Delta_\epsilon(a_j|s_j)\kappa_j(s_{j+1}|g_{j-1},s_j,a_j)\1\set{s_0,\ds,s_{j-1}\neq y, s_j = y}  \gamma^{k-j}r(s_k,a_k)\\
        &\stackrel{(ii)}{=} \sum_{j=1}^\infty\sum_{k=j}^{\infty}  \sum_{g_k\in \bd{G}_k}\Delta_\epsilon(a_j|s_j)p_\epsilon(s_{j+1}|s_j,a_j)\1\set{s_0,\ds,s_{j-1}\neq y, s_j = y}  \gamma^{k-j}r(s_k,a_k)
    \end{aligned}$$
    Here $(i)$ applies Fubini's theorem leveraging the positivity of the summand, $(ii)$ follows from the definition of $\kappa$ in \eqref{eqn:to_use_def_kappahd}. From $(ii)$, reversing the previous equalities, we have that 
    \begin{equation}\label{eqn:to_use_second_half_eq}\begin{aligned}
        &E^{\Delta_\epsilon,\kappa}_x \sum_{k=\tau_y}^{\infty} \gamma^{k-\tau_y} r(X_k,A_k) \\
        &= E^{\Delta_\epsilon,p_\epsilon}_x \sum_{k=\tau_y}^{\infty} \gamma^{k-\tau_y} r(X_k,A_k)\\
        &= E^{\Delta_\epsilon,p_\epsilon}_x E^{\Delta_\epsilon,p_\epsilon}_x \sqbkcond{\sum_{k=\tau_y}^{\infty} \gamma^{k-\tau_y} r(X_k,A_k)}{\cH_{\tau_y}}\\
        &\stackrel{(i)}{=} E^{\Delta_\epsilon,p_\epsilon}_y \sum_{k=0}^{\infty} \gamma^{k} r(X_k,A_k)\\
        &= v_\gamma^{\Delta_\epsilon,p_\epsilon}(y).
    \end{aligned}
    \end{equation}
    where $(iv)$ is because $p_\epsilon$ is a stationary policy and, under $E_x^{\Delta_\epsilon,p_\epsilon}$, $\set{X_t,A_t:t\geq 0}$ is strong Markov. 

    Therefore, the bound in \eqref{eqn:to_use_vkappa_hd_bd} and equalities \eqref{eqn:to_use_second_half_eq} and \eqref{eqn:to_use_E_kappa_tau_y} imply that for all $x\in S$
    \begin{equation}\label{eqn:wc_vx-vy_ub}
    \begin{aligned}
    &v_{\gamma}^{\Delta_\epsilon,\kappa}(x)-v_{\gamma}^{\Delta_\epsilon,p_\epsilon}(y) \\
    \leq& E_x^{\Delta_\epsilon,q}\tau_y + v_\gamma^{\Delta_\epsilon,p_\epsilon}(y) -v_{\gamma}^{\Delta_\epsilon,p_\epsilon}(y) \leq \frac{|S|}{\delta'} .
    \end{aligned}
    \end{equation}
    where the last inequality follows from Lemma \ref{lemma:wc_bd_E_hitting_time} and $y\in C_{\Delta_k}$. In particular, since $x\in S$ and $y\in C_{\Delta_k}$ in \eqref{eqn:wc_vx-vy_ub} are arbitrary, 
    \begin{equation}\label{eqn:to_use_xi1_ub_final}
       \xi_1 =  v_{\gamma}^{\Delta_\epsilon,\kappa}(s_\vee)- v_{\gamma}^{\Delta_\epsilon,p_\epsilon }(y_\wedge)\leq \frac{|S|}{\delta'}.  
    \end{equation}

\begin{remark}\label{rmk:cmpct_P_not_needed_2}
As noted in Remark \ref{rmk:cmpct_P_not_needed} and \ref{rmk:cmpct_P_not_needed_1}, the compactness of the adversarial ambiguity set $\cP$ is not required for Lemma \ref{lemma:wc_bd_E_hitting_time}, nor is it used in the proof of the upper bound for $\xi_1$. Hence, $\xi_1$ admits the same upper bound without assuming compactness of $\cP_s$ for $s\in S$. Therefore, we conclude the existence of a solution. 

In contrast, compactness is essential for our proof of the following upper bound on $\xi_2$, since that argument relies on the properties of $K_\Delta$.
\end{remark}

\subsection*{Upper-Bounding $\xi_2$}

Similar to $\xi_1$, we will upper bound $\xi_2$ by the expected hitting time of $C_{\Delta_k}$. Specifically, let $T_k = \inf\set{t\geq 0:X_{t} \in C_{\Delta_k}} $. By definition,
$$v_\gamma^{\Delta_\epsilon,p_\epsilon}(s_\wedge) = E_{s_\wedge}^{\Delta_\epsilon,p_\epsilon}\sum_{t=0}^{T_k-1}\gamma^t r(X_t,A_t) + E_{s_\wedge}^{\Delta_\epsilon,p_\epsilon}\sum_{t=T_k}^{\infty}\gamma^t r(X_t,A_t)$$

Note that by the same argument as in \eqref{eqn:to_use_second_half_eq}
\[\begin{aligned}
&E_{s_\wedge}^{\Delta_\epsilon,p_\epsilon}\sum_{t=T_k}^{\infty}\gamma^t r(X_t,A_t) \\
= &E^{\Delta_\epsilon,p_\epsilon}_{s_\wedge} \gamma ^{T_k}E^{\Delta_\epsilon,p_\epsilon}_{s_\wedge} \sqbkcond{\sum_{t=T_k}^{\infty} \gamma^{t-T_k} r(X_t,A_t)}{\cH_{T_k}}\\
=& E^{\Delta_\epsilon,p_\epsilon}_{s_\wedge} \gamma ^{T_k} v_\gamma^{\Delta_\epsilon,p_\epsilon}(X_{T_k})\\
\geq& v_\gamma^{\Delta_\epsilon,p_\epsilon}(y_\wedge) E^{\Delta_\epsilon,p_\epsilon}_{s_\wedge} \gamma^{T_k} 
\end{aligned}\]
where the last inequality follows from the choice of $y_\wedge$ in \eqref{eqn:to_use_def_y_wedge}. Therefore, 
$v_\gamma^{\Delta_\epsilon,p_\epsilon}(s_\wedge)  \geq v_\gamma^{\Delta_\epsilon,p_\epsilon}(y_\wedge) E^{\Delta_\epsilon,p_\epsilon}_{s_\wedge} \gamma ^{T_k} $ and hence
\begin{equation}\label{eqn:to_use_vywedge-vswedge_ub}
 \begin{aligned}&v_\gamma^{\Delta_\epsilon,p_\epsilon}(y_\wedge) - v_\gamma^{\Delta_\epsilon,p_\epsilon}(s_\wedge) \\
&\leq v_\gamma^{\Delta_\epsilon,p_\epsilon}(y_\wedge)  E^{\Delta_\epsilon,p_\epsilon}_{s_\wedge}[1-\gamma^{T_k}]\\
&=  (1-\gamma)v_\gamma^{\Delta_\epsilon,p_\epsilon}(y_\wedge)  E^{\Delta_\epsilon,p_\epsilon}_{s_\wedge}\frac{1-\gamma^{T_k}}{1-\gamma}\\
&\stackrel{(i)}{\leq} E^{\Delta_\epsilon,p_\epsilon}_{s_\wedge}\sum_{t=0}^{T_k-1}\gamma^t\\
&\leq E^{\Delta_\epsilon,p_\epsilon}_{s_\wedge} T_k
\end{aligned}
\end{equation}
where $(i)$ follows from $0\leq r\leq 1$ and hence $0\leq v_\gamma^{\Delta_\epsilon,p_\epsilon} \leq 1/(1-\gamma)$ as well as $\sum_{t=0}^{k-1} \gamma^t= (1-\gamma^k)/(1-\gamma). $


Note that $E^{\Delta_\epsilon,p_\epsilon}_{s_\wedge} T_k$ implicitly depends on $\gamma$ via $\Delta_\epsilon$ and $p_\epsilon$. To provide a uniform upper bound, we consider $$E^{\Delta_\epsilon,p_\epsilon}_{s_\wedge} T_k \leq \max_{k\leq |B|}\sup_{\Delta\in \overline{G_{\Delta_k}},\;p\in\cP}E_{s_\wedge}^{\Delta,p}T_{k} \leq \max_{k\leq |B|}\sup_{\Delta\in \overline{G_{\Delta_k}},\;p\in\cP}\max_{x\in C_{\Delta_k}^c}E_{x}^{\Delta,p}T_{k},$$
where the last inequality follows from the fact that if $x\in C_{\Delta_k}$, $T_k = 0$ w.p.1.

So, we only need to consider $k\leq |B|$ with $C^c_{\Delta_k}\neq \varnothing$. To proceed, we recall the properties of $K_{\Delta_k}$ defined by \eqref{eqn:to_use_within_K_eval_bd}. Moreover, by construction $\overline{G_{\Delta}}\subseteq \overline{O_\Delta}\cap \overline{K_\Delta}\subseteq \overline{K_\Delta}$. So, we have that by Lemma \ref{lemma:construct_K_Delta}, for all $\Delta\in \overline{G_{\Delta_k}}$, $$\sup_{p\in\cP} e^\top (I-M_{\Delta,p}^{\Delta_k})\inv e \leq  \sup_{p\in\cP}e^\top(I-M_{\Delta_k,p}^{\Delta_k})\inv e + 1  < \infty. $$
Therefore, by the first transition analysis argument, 
\[\begin{aligned}
   & \max_{k\leq |B|}\sup_{\Delta\in \overline{G_{\Delta_k}},\;p\in\cP}\max_{x\in C_{\Delta_k}^c} E_x ^{\Delta,p} T_k  \\
    = &\max_{k\leq |B|}\sup_{\Delta\in \overline{G_{\Delta_k}},\;p\in\cP}\max_{x\in C_{\Delta_k}^c} [(I-M_{\Delta,p}^{\Delta_k})\inv e ](x)\\
    \leq& \max_{k\leq |B|}\sup_{\Delta\in \overline{G_{\Delta_k}}}\sup_{p\in\cP}e^\top (I-M_{\Delta,p}^{\Delta_k})\inv e \\
    \leq& 1 + \sum_{k\leq |B|}\sup_{p\in\cP}e^\top (I-M_{\Delta_k,p}^{\Delta_k})\inv e,
\end{aligned}\] 
where the expression in the last line is finite (by Lemma \ref{lemma:construct_K_Delta}) and independent of $\epsilon$ and $\gamma$. Therefore, going back to \eqref{eqn:to_use_vywedge-vswedge_ub}, we conclude that

\begin{equation}\label{eqn:to_use_xi2_ub_final}
\begin{aligned}
\xi_2& = v_\gamma^{\Delta_\epsilon,p_\epsilon}(y_\wedge) - v_\gamma^{\Delta_\epsilon,p_\epsilon}(s_\wedge) \\ 
&\leq E^{\Delta_\epsilon,p_\epsilon}_{s_\wedge} T_k \leq 1 + \sum_{k\leq |B|}\sup_{p\in\cP}e^\top (I-M_{\Delta_k,p}^{\Delta_k})\inv e 
\end{aligned}
\end{equation}
is uniformly bounded in $\epsilon$ and $\gamma$. 

\subsection*{Concluding Theorem \ref{thm:wc_ctrl}}
Combining \eqref{eqn:to_use_vspan_decomp}, \eqref{eqn:to_use_xi1_ub_final}, and \eqref{eqn:to_use_xi2_ub_final} yields
$$\spnorm{v_\gamma^*}\leq \frac{|S|}{\delta'} + 1 + \sum_{k\leq |B|}\sup_{p\in\cP}e^\top (I-M_{\Delta_k,p}^{\Delta_k})^{-1}e + 2\epsilon. $$
By Lemmas \ref{lemma:wc_bd_E_hitting_time} and \ref{lemma:construct_K_Delta} and $\epsilon>0$ can be arbitrarily small, it follows that $\spnorm{v_\gamma^*}$ is uniformly bounded for all $\gamma\in(0,1)$. Together with Theorem \ref{thm:existence_of_sol}, this establishes Theorem \ref{thm:wc_ctrl}.

\section{Bellman Optimality for the HD-S Case}\label{sec:HD_S}
\newcommand{\piRL}{\pi_{\mathrm{RL}}}

As noted in Remark \ref{rmk:left_out_HDS}, in this section, we show a surprising result that, under a weak communication assumption, the average reward for a history-dependent controller against a stationary adversary corresponds to the solution of the Bellman equation with the inf-sup ordering \eqref{eqn:inf_sup_eqn_const_gain}, rather than its original form. Moreover, $\epsilon$-optimal average rewards can be attained by online reinforcement learning (RL) policies, which are inherently history-dependent.

\begin{proposition}    
\label{prop:hd_eps_opt_policy}
If $\set{\delta_a:a\in A}\subseteq\cQ$ and the adversary is weakly communicating, then for each $\epsilon > 0$, there exists a history-dependent RL policy $\piRL\in \PiH$ s.t. 
$$ 0 \leq \underline{\alpha}(\mu,\PiH,\KS) - \inf_{\kappa\in\KS}\underline{\alpha} (\mu,\piRL,\kappa) \stackrel{(ii)} {\leq}\inf_{\kappa\in\KS}\sup_{\pi\in\PiS}\underline{\alpha}(\mu,\pi,\kappa) - \inf_{\kappa\in\KS}\underline{\alpha} (\mu,\piRL,\kappa) \leq \epsilon. $$
The same result holds true if $\underline{\alpha}$ is replaced with $\overline{\alpha}$. 
\end{proposition}

Note that in the second inequality, we swap to inf-sup. The proof of Proposition \ref{prop:hd_eps_opt_policy} is deferred to Appendix \ref{sec:proof_prop_RL_policy}, where we instantiate the RL policy $\pi_{\mrm{RL}}$ with the online algorithm in \citet{zhang2023amdp_regret}. This choice is illustrative rather than exclusive: any online RL algorithm that (i) uses only deterministic actions (consistent with the assumption $\{\delta_a:a\in A\}\subseteq\cQ$) and (ii) achieves sublinear regret can be employed to obtain $\epsilon$-optimality for any prescribed $\epsilon>0$.

\begin{theorem}\label{thm:value_HD_S}
If $\set{\delta_a:a\in A}\subseteq\cQ$, then so long as the adversary is weakly communicating, $$\underline{\alpha}(\mu,\PiH,\KS) = \inf_{\kappa\in\KS}\sup_{\pi\in\PiH}\underline\alpha(\mu,\pi,\kappa)=\inf_{\kappa\in\KS}\sup_{\pi\in\PiS}\underline\alpha(\mu,\pi,\kappa).$$
The same result holds true if $\underline{\alpha}$ is replaced with $\overline{\alpha}$. Moreover, if \eqref{eqn:inf_sup_eqn_const_gain} admits a solution pair $(u',\alpha')$, then 
$$\underline{\alpha}(\mu,\PiH,\KS) \;=\; \overline{\alpha}(\mu,\PiH,\KS) \;=\; \alpha'.$$ 
\end{theorem}

Intuitively, when the stationary adversary is weakly communicating, a history-dependent controller policy can adaptively “learn’’ the adversary policy through online reinforcement learning. Importantly, this learning process doesn't affect the long-run average performance of the controller policy, hence achieves the inf-sup value. 

In particular, this implies that if solutions exist for \eqref{eqn:r_bellman_eqn_const_gain} and \eqref{eqn:inf_sup_eqn_const_gain}, but the corresponding gains $\alpha^*$ and $\alpha'$ do not coincide, then stationary optimality cannot be expected for the robust optimal control problems $\underline{\alpha}(\mu,\PiH,\KS)$ and $\overline{\alpha}(\mu,\PiH,\KS)$. A converse of this is also true. This is summarized in the following Corollary \ref{cor:iff}. 


\begin{corollary} \label{cor:iff}
    Assume that both the controller and the adversary are weakly communicating, $\set{\delta_a:a\in A}\subseteq\cQ$, and that $\cQ$ and $\cP_s,\;s \in S$ are compact. If the adversary’s policy class is stationary, i.e., $\mrm K=\KS$, then the optimality gap between using history-dependent $\Pi=\PiH$ and stationary $\Pi=\PiS$ control policies is $\alpha'-\alpha^*$. In particular, stationary policies are optimal for a history-dependent controller if and only if $\alpha'=\alpha^*$.
\end{corollary}

Next, using a robust MDP instance adapted from \citet{wang_foundation_2024}, we demonstrate the necessity of history-dependent policies for achieving optimal robust control under the average-reward criterion when facing a stationary adversary. In this example, both the adversary and the controller are communicating. 

We show that the optimal robust average reward is characterized by the inf–sup equation \eqref{eqn:inf_sup_eqn_const_gain}, rather than the robust Bellman equation \eqref{eqn:r_bellman_eqn_const_gain}; in particular, $\alpha^* \neq \alpha'$. Furthermore, we present a simple history-dependent policy that attains this optimal value. The dynamics of this robust MDP instance is illustrated in Figure \ref{example:HD-S}.

\begin{figure}[htb]
	\centering
	\subfigure[$p^{(1)}$]{
		\scalebox{1}{
			\begin{tikzpicture}
				\node[state]                               (A) {\text{I}};
				\node[state,right=of A,yshift=1cm,xshift=0.5cm]                   (B) {G};
				\node[state,right=of A,yshift=-1cm,xshift=0.5cm]                    (C) {B};

				\path(A)   edge[blue] node[above] {}  (B);
				
				\path (A)   edge[red] node[above] {}   (C);
				\path 	(B.north)   edge[bend right=60]    (A.north);
				\path 	(C.south)   edge[bend left=60]    (A.south);

	\end{tikzpicture}}} \hspace{1in}
	\subfigure[$p^{(2)}$]{
		\scalebox{1}{
			\begin{tikzpicture}
				\node[state]                               (A) {\text{I}};
				\node[state,right=of A,yshift=1cm,xshift=0.5cm]                   (B) {G};
				\node[state,right=of A,yshift=-1cm,xshift=0.5cm]                    (C) {B};

				\path(A)   edge[red] node[above] {}  (B);
				
				\path (A)   edge[blue] node[above] {}   (C);
				\path 	(B.north)   edge[bend right=60]    (A.north);
				\path 	(C.south)   edge[bend left=60]    (A.south);

	\end{tikzpicture}}} 
	\caption{The dynamics of an average-reward robust MDP instance where $\alpha^*\neq\alpha'$. The red and blue arrows represent the dynamics under actions $a_1$ and $a_2$, respectively.}
    \label{example:HD-S}
\end{figure}

The state space is $S = \set{\text{I},\text{G},\text{B}}$, where I stands for the initial state, G is the good state, and B is a bad state. We consider the following reward function $r$, which does not depend on the control actions 
$$r(\text{I})=0,\; r(\text{G})=1,\;r(\text{B})=-1.
$$

In I, two actions ${A} = \{a_1,a_2\}$ lead to different dynamics, whereas in G and B, taking different two actions will not change the dynamics. In particular, in state I, the S-rectangular adversary can choose from $\cP_\text{I}:=\set{p_\text{I}^{(1)},p_\text{I}^{(2)}}$. In kernel form,
$$
	p^{(1)}(\text{B}|\text{I},a_{1}) =1,\; p^{(1)}(\text{G}|\text{I},a_{2})=1,  \text{ and }
	p^{(2)}(\text{G}|\text{I},a_{1}) =1, p^{(2)}(\text{B}|\text{I},a_{2})=1.
$$
Here, \( p^{(1)}(B|\text{I},a_{1}) =1 \) means that in state \(\text{I}\), if the controller selects action \( a_1 \) and the adversary chooses \( p_\text{I}^{(1)} \), the MDP transitions to state B with probability 1. The other transition probabilities are interpreted in the same way. We assume $\cQ = \cP(\cA)$.

Note that in this robust MDP instance, both the controller and the adversary are communicating.  
Specifically, fixing any action (or randomized policy) chosen by the controller, state G (respectively B) can reach state B (respectively G) under at least one of the kernels $p^{(i)}$, $i=1,2$, while state I is always recurrent.  
Similarly, fixing any adversarial kernel $p^{(i)}$, $i=1,2$, there always exists a control action that enables a transition from G (B) to B (G).

In this example, it is not hard to see that if the controller can only use stationary policies and the adversary plays second, the best stationary strategy of the controller is to randomize the two actions with probability $1/2$. Therefore, one can easily verify that $\overline\alpha(\text I,\PiS,\KS) = 0$, and a solution to the robust Bellman equation (\ref{eqn:r_bellman_eqn_const_gain}) is
$$
\alpha^*=0,\;u^*(\text{I})=0,\;u^*(\text{G})=1,\;u^*(\text{B})=-1.
$$
In particular, Theorem \ref{thm:general_dpp} holds and  
    $$
    \begin{aligned}
    0&=\overline\alpha(\text{I},\PiH,\KH) = \underline\alpha(\text{I},\PiH,\KH) \\
    &=\overline\alpha(\text{I},\PiS,\KH) = \underline\alpha(\text{I},\PiS,\KH) =\overline\alpha(\text{I},\PiS,\KS) = \underline\alpha(\text{I},\PiS,\KS). 
    \end{aligned}$$
    
On the other hand, if the controller plays second, it is always able to exploit the knowledge of the adversary's choice $p^{(i)}$, and counter with $a_{(i\text{ mod }2)+1}$. So, one would expect that $\alpha' =0.5$, as, in the long run, the Markov chain will spend half of the time in state I and the other half of the time in G. With this intuition, it is not hard to verify that the inf-sup equation  \eqref{eqn:inf_sup_eqn_const_gain} is solved by
$$
\alpha'=0.5,\;u'(\text{I})=0,\;u'(\text{G})=0.5,\;u'(\text{B})=-1.5.
$$
Then, Theorem \ref{thm:value_HD_S} indicates $\underline{\alpha}(\text{I},\PiH,\KS) = \overline{\alpha}(\text{I},\PiH,\KS) = 0.5.$ In particular, as Corollary \ref{cor:iff} suggests, stationary policies cannot be optimal for the controller for this MDP instance.

Moreover, we can construct a simple optimal history-dependent controller policy that achieves the optimal gain of $0.5$ as follows.  
At time $0$, starting from the initial state I, the controller selects an arbitrary action, say $a_2$.

If state G is observed at time $1$, the controller can infer that the adversary has selected kernel $p^{(1)}$. Consequently, the controller continues to choose $a_2$ for all subsequent time steps. This induces a deterministic Markov chain alternating between I and G, yielding a long-run average reward of $0.5$.

Conversely, if state B is observed at time $1$, the controller infers that the adversary has selected kernel $p^{(2)}$. The controller then switches to action $a_1$ for all subsequent time steps, again inducing a deterministic Markov chain alternating between I and G, and thus achieving the same long-run average reward of $0.5$.

\section{Examples in Operations Environments}\label{section:OR_example}
In this section, we discuss two classes of robust stochastic control environments that are highly relevant to operations research. These examples illustrate a common theme in operations environments: under-optimized feasible decisions often cannot guarantee system stability, whereas good decisions can. In particular, different decisions can lead to substantially different recurrent classes, so existing dynamic programming theory based on uniform accessibility of a recurrent state \citep{wang2023avg_unichain_dp} for average-reward robust MDPs does not apply. In contrast, our framework applies naturally to this type of setting because it only requires weak communication, which is widely applicable when stability can be induced by good decisions.

\subsection{Dynamic Service Allocation in Queuing Environments}

Consider a queuing system with $m$ queues, each with a waiting room of capacity $k$. Customers do not join the system if the waiting room is full. A state is denoted by $s=(q_1,q_2,\ds,q_m),$ where \(q_i\) is the number of customers in the $i$th queue. There are $n$ servers in the system, each of which can be assigned to serve one queue at a time. The possible service allocation vectors are $A = \set{(a_1,a_2,\ds,a_m)\in\N^m:\sum_{i=1}^m a_i = n}$, 
where $a_i$ is the number of servers assigned to serve queue $i$. We assume that the service takes one unit of time. All customers arrive at the beginning of each time unit, and decisions are made at the end of each time unit. 

Let $N_{i,t}$, $Q_{i,t}$, and $A_{i,t}$ denote the number of new arrivals, the queue length, and the number of allocated servers associated with queue $i$ at time $t$, respectively. In each time period, new customers arrive to queue $i$, and the servers
allocated to queue $i$ can serve at most $a_i$ customers, for
$i=1,\ldots,m$. After arrivals, the effective queue length is truncated by the capacity $k$ and is therefore $\min\{Q_{i,t}+N_{i,t},k\}. $ Thus, the queue-length dynamics are
\[
Q_{i,t+1}
=(
\min\{Q_{i,t}+N_{i,t},k\}-A_{i,t})_+, 
\]
for $i=1,2,\ldots,m$.

While our theory applies in the general setting, to illustrate the challenge, we consider the simple setting where $m = k = 2$ and $n = 1$. The state space is thus $S=\{(q_1,q_2): q_1,q_2 \in \{0,1,2\}\}.$ There is only one server that serves one of the two queues. 
We assume that at most one new customer arrives at each queue during each time unit; i.e., the four possible joint arrival outcomes are \((0,0)\), \((0,1)\), \((1,0)\), and \((1,1)\), all with positive probability.
Figure \ref{fig:twoqueue-transitions} below illustrates the transition structures under the two possible service allocations, where the server serves either queue 1 or queue 2.

\begin{figure}[ht]
\centering

\subfigure[The server only serves queue 1.]{
\centering
\begin{tikzpicture}[
    >=Stealth,
    node distance=2cm,
    every node/.style={font=\normalsize},
    state/.style={circle,draw,minimum size=10mm,inner sep=0pt},
    edge/.style={->,thin},
    loopedge/.style={->,thin},
    scale=0.9,
    transform shape
]

\node[state] (s02) at (0,4)  {\(0,2\)};
\node[state] (s12) at (2,4)  {\(1,2\)};
\node[state] (s22) at (4,4)  {\(2,2\)};

\node[state] (s01) at (0,2)  {\(0,1\)};
\node[state] (s11) at (2,2)  {\(1,1\)};
\node[state] (s21) at (4,2)  {\(2,1\)};

\node[state] (s00) at (0,0)  {\(0,0\)};
\node[state] (s10) at (2,0)  {\(1,0\)};
\node[state] (s20) at (4,0)  {\(2,0\)};

\draw[loopedge] (s00) edge[loop below] (s00);
\draw[loopedge] (s10) edge[loop below] (s10);

\draw[loopedge] (s01) edge[loop left] (s01);
\draw[loopedge] (s11) edge[loop, out=60, in=30, looseness=8] (s11);

\draw[loopedge] (s02) edge[loop above] (s02);
\draw[loopedge] (s12) edge[loop above] (s12);

\draw[edge] (s00) -- (s01);

\draw[edge] (s10) -- (s00);
\draw[edge] (s10) -- (s01);
\draw[edge] (s10) -- (s11);

\draw[edge] (s20) -- (s10);
\draw[edge] (s20) -- (s11);

\draw[edge] (s01) -- (s02);

\draw[edge] (s11) -- (s01);
\draw[edge] (s11) -- (s02);
\draw[edge] (s11) -- (s12);

\draw[edge] (s21) -- (s11);
\draw[edge] (s21) -- (s12);

\draw[edge] (s12) -- (s02);

\draw[edge] (s22) -- (s12);

\end{tikzpicture}
}\centering
\subfigure[The server only serves queue 2.]{

\begin{tikzpicture}[
    >=Stealth,
    node distance=2cm,
    every node/.style={font=\normalsize},
    state/.style={circle,draw,minimum size=10mm,inner sep=0pt},
    edge/.style={->,thin},
    loopedge/.style={->,thin},
    scale=0.9,
    transform shape
]

\node[state] (s02) at (0,4)  {\(0,2\)};
\node[state] (s12) at (2,4)  {\(1,2\)};
\node[state] (s22) at (4,4)  {\(2,2\)};

\node[state] (s01) at (0,2)  {\(0,1\)};
\node[state] (s11) at (2,2)  {\(1,1\)};
\node[state] (s21) at (4,2)  {\(2,1\)};

\node[state] (s00) at (0,0)  {\(0,0\)};
\node[state] (s10) at (2,0)  {\(1,0\)};
\node[state] (s20) at (4,0)  {\(2,0\)};

\draw[loopedge] (s00) edge[loop below] (s00);
\draw[loopedge] (s10) edge[loop below] (s10);
\draw[loopedge] (s20) edge[loop below] (s20);

\draw[loopedge] (s01) edge[loop left] (s01);
\draw[loopedge] (s11) edge[loop, out=60, in=30, looseness=8] (s11);
\draw[loopedge] (s21) edge[loop right] (s21);

\draw[edge] (s00) -- (s10);

\draw[edge] (s10) -- (s20);

\draw[edge] (s01) -- (s00);
\draw[edge] (s01) -- (s10);
\draw[edge] (s01) -- (s11);

\draw[edge] (s11) -- (s10);
\draw[edge] (s11) -- (s20);
\draw[edge] (s11) -- (s21);

\draw[edge] (s21) -- (s20);

\draw[edge] (s02) -- (s01);
\draw[edge] (s02) -- (s11);

\draw[edge] (s12) -- (s11);
\draw[edge] (s12) -- (s21);

\draw[edge] (s22) -- (s21);

\end{tikzpicture}
}

\caption{Transition diagrams for different actions in the example of dynamic service allocation in queuing environments.}
\label{fig:twoqueue-transitions}
\end{figure}

From Figure~\ref{fig:twoqueue-transitions}, it is easy to see that under the policy serving only queue $1$, the only recurrent state is $(0,2)$, while under the policy serving only queue $2$, the only recurrent state is $(2,0)$. Hence, different policies lead to disjoint recurrent classes, and prior work assuming uniform accessibility cannot be applied.

On the other hand, the corresponding transition kernel $p$ is weakly communicating, with transient state $(2,2)$ and communicating class $C_p=S\backslash\set{(2,2)}$. Indeed, serving the first queue can move the state up and to the left, while serving the second queue can move the state down and to the right in the transition diagram. By combining these actions, one can move from any state to any other state in $C_p$, while the state $(2,2)$ is always transient. Thus, if all transition kernels in $\cP$ share the same transition graph and differ only in their transition probabilities, as may occur under mutual absolute continuity (e.g., with a sufficiently small ambiguity radius), then the adversary is weakly communicating. Hence, Theorem~\ref{thm:wc_adv} guarantees that the inf-sup equation \eqref{eqn:inf_sup_eqn_const_gain} admits a solution. Furthermore, under additional convexity assumptions or SA-rectangularity, Theorem~\ref{thm:suff_cond_for_swap} further guarantees the existence of a solution to the constant-gain robust Bellman equation \eqref{eqn:r_bellman_eqn_const_gain}.


We also note that, in service-allocation settings, S- and SA-rectangular ambiguity models correspond to different assumptions about how ambiguity arises in the system dynamics. For example, consider a queueing system with homogeneous service rates in which model uncertainty arises solely from customer arrivals. In this setting, an S-rectangular ambiguity model may be more appropriate, as customers' arrival processes are exogenous to the service-allocation action. In comparison, SA-rectangular ambiguity may grant the adversary unreasonable power from a modeling perspective, allowing it to vary transition distributions independently across actions and thereby potentially alter the service dynamics.


In contrast, an SA-rectangular model can be appropriate when the ambiguity set reflects unmodeled service heterogeneity. For instance, differences in server skills, processing efficiency, or equipment quality may cause the effective transition dynamics to depend on the assignment action itself. In such cases, changing the server allocation can alter the system dynamics in ways that are difficult to predict in advance, making action-dependent ambiguity a plausible modeling choice.


\subsection{Inventory Control}
We consider a simple Markovian multi-item lost-sales inventory control problem with state space $S=\set{(x_1,\ds,x_m)\in\N^m:0\leq x_i\leq b_i,\ \forall i},$
where $x_i$ and $b_i$ denote the inventory level and capacity for item $i$, respectively. A natural action space is
$A:=\set{(a_1,\ds,a_m)\in\N^m:0\leq a_i\leq b_i,\ \forall i},$
where $a_i$ represents the number of item $i$ ordered and delivered at the beginning of day $t$. For simplicity, we assume no lead time between ordering and the arrival of replenishment inventory. The state transition is governed by
$$
X_{t+1}=\crbk{\min\set{X_t+A_t,b}-D_{t+1}}_+,
$$
where $\set{D_t\in\N^m:t\geq 1}$ is the demand process. 

Similar to the service allocation example, non-optimized decisions need not induce the same long-run system behavior. Indeed, consider the policy $\pi^{(j)}$ that orders $a_j=b_j$ and $a_i=0$ for all $i\neq j$ and all $s\in S$. Suppose the demand satisfies $P(D_i=0) < 1$ for all $i\neq j$. Then, under \(\pi^{(j)}\), every recurrent state satisfies \(x_i=0\) for all \(i\neq j\). Moreover, if \(D_j<b_j\) almost surely, then every recurrent state also satisfies \(x_j>0\). Consequently, the policies $\{\pi^{(j)}:j=1,\ldots,m\}$ do not share a common recurrent state, and therefore uniform accessibility fails.

In comparison, the irreducibility condition in Corollary \ref{cor:exists_irreducible}.2 is easy to verify under the following mild assumption: for every \(p\in\cP\) the induced demand satisfies $P(D=\mathbf{0}_m)>0$ and $P(D_i > 0) > 0$ for all $i$, where $\mathbf{0}_m$ is an $m$-dimensional vector of zeros.  This is a natural non-degeneracy condition in inventory control applications. To see this, consider the policy that selects actions from $A$ uniformly at random. Then, the induced Markov chain is irreducible. Indeed, the policy can move from the state \((0,0,\ds,0)\) to any state $x\in S$ with positive probability when the demand is zero, and, starting from any non-zero state, it can move the state back to 0 with positive probability by consecutively selecting the 0 action. Therefore, Corollary \ref{cor:exists_irreducible}.2 guarantees that the inf-sup equation \eqref{eqn:inf_sup_eqn_const_gain} admits a solution, and  Theorem \ref{thm:suff_cond_for_swap} further ensures the existence of a solution to the robust Bellman equation \eqref{eqn:r_bellman_eqn_const_gain} if convexity or SA-rectangularity holds.

We remark that, compared to SA-rectangular models, S-rectangular/action-unaware adversarial models are more suitable for robust inventory control, due to the natural decoupling between the current action and demand behavior. Multiple S-rectangular formulations exist and are supported by both theoretical intuition and empirical evidence \citep{ho2018fast_bellman,wang2025DRSC,li2026SRec}. To illustrate, we consider the current-action-unaware formulation adapted from \citet{wang2025DRSC}. Let $\cP = \bigtimes_{s\in S}\cP_s$ be given by $\cP_s:=\set{p_{s,\psi}:\psi\in\Psi}$, where
$$
p_{s,a,\psi}
:=
\text{Law}_{D\sim\psi}\crbk{\crbk{\min\set{s+a,b}-D}_+}
$$
is the pushforward measure induced by the state transition dynamics. The ambiguity set $\Psi$ over the law of $D$ can take the form
$\Psi:=\set{\psi:d(\psi\|\psi_0)\leq\delta},$
where $\psi_0$ is the nominal law, $d$ is a metric or divergence, and $\delta$ is the ambiguity radius. Therefore, if every $\psi\in\Psi$ satisfies $P_{D\sim\psi}(D =\mathbf{0}_m) > 0$ and $P_{D\sim\psi}(D_i > 0) > 0$ for all $i$, then the irreducibility condition is verified. In particular, this holds when $P_{D\sim\psi_0}(D =\mathbf{0}_m)) > 0$, $P_{D\sim\psi_0}(D_i > 0) > 0,\ \forall i$,  and the ambiguity radius is reasonably small.
\section*{Acknowledgments}
N. Si’s research is generously supported by the Early Career Scheme [Grant 26210125] from the Hong Kong
Research Grants Council and HKUST Li \& Fung Supply Chain Institute Research Grant 2025.
\bibliographystyle{apalike}
\bibliography{proof_bibs,DR_MDP,mdps}
\newpage

\appendix
\appendixpage
\section{Proofs for Section \ref{sec:bellman_opt}}\label{sec:proof_bellman_eqn_opt}

Recall that the history
\[
    \bd H_t
    :=
    (S \times A)^t\times S
    =
    \bigl\{
    (s_0,a_0,\ldots,s_{t-1},a_{t-1},s_t):
    s_k\in S,\ a_k\in A
    \bigr\}.
\]
We also define the random element $H_t:\Omega\ra \bd{H}_t$ by point evaluation $H_t(\omega) = h_t$, and the $\sigma$-field $\cH_t:=\sigma(H_t)$. Next, we define $\set{\bd{G}_t:t\geq 0 }$ by 
$$\bd{G}_t := (S \times A)^{t+1} =
    \bigl\{
    (s_0,a_0,\ldots,s_{t-1},a_{t-1},s_t,a_t):
    s_k\in S,\ a_k\in A
    \bigr\}.$$
Note that $g_t$ is the concatenation of the history $h_t$ with the controller's action at time $t$, i.e., $g_t = (h_t,a_t)$, where $h_t\in\bd{H}_t$. Also, define the random element $G_t:\Omega\ra \bd{G}_t$ by point evaluation $G_t(\omega) = g_t$, and $\cG_t:=\sigma(G_t)$. 

To prove the main theorems in Section \ref{sec:bellman_opt}, we introduce an important technical tool. 

\begin{proposition}\label{prop:mg}
For any function $f:S\ra\R$ and any pair of policies $\pi\in\PiH$ and $\kappa\in\KH$, define the process $$M_{f,n}^{\pi,\kappa} = \sum_{k=1}^n\left[  f(X_k) - \sum_{a,s'}\pi_{k-1}(a|H_{k-1})\kappa_{k-1}(s'|H_{k-1},a)f(s') \right]. $$  
Then, $M_{f,n}^{\pi,\kappa}$ is a $\cH_n,P_\mu^{\pi,\kappa}$-Martingale. 
\end{proposition}

\subsection{Proof of Proposition \ref{prop:mg}}
\begin{proof}
\par It suffices to check $E[M_{f,k}^{\pi,\kappa} - M_{f,k-1}^{\pi,\kappa}|\cH_{k-1}] = 0$. 
    
\par We see that the conditional distribution of $(A_{k-1},X_k)$ given $H_{k-1}$ is determined by $\pi_{k-1}$ and $\kappa_{k-1}$. So, 
$$E_\mu^{\pi,\kappa}\sqbkcond{f(X_k) }{\cH_{k-1}} = \sum_{a,s'}\pi_{k-1}(a|H_{k-1})\kappa_{k-1}(s'|H_{k-1},a)f(s').$$ Also, note that $$M_{f,k}^{\pi,\kappa} - M_{f,k-1}^{\pi,\kappa} = f(X_k) - \sum_{a,s'}\pi_{k-1}(a|H_{k-1})\kappa_{k-1}(s'|H_{k-1},a)f(s').$$ This completes the proof. 
\end{proof}

\subsection{Proof of Theorem \ref{thm:general_dpp}}\label{section:proof:thm:general_dpp}
\begin{proof}

We first prove the second claim that any solution $(u,\alpha)$ to \eqref{eqn:r_bellman_eqn_const_gain} satisfies $\alpha = \alpha^*$, assuming \eqref{eqn:general_dpp} holds. 
If $(u,\alpha)$ is any solution to \eqref{eqn:r_bellman_eqn_const_gain}, then by \eqref{eqn:general_dpp}, $\alpha = \overline\alpha(\mu,\PiH,\KH)$. On the other hand, by \eqref{eqn:general_dpp}, $ \overline\alpha(\mu,\PiH,\KH) = \alpha^*$. So, $\alpha = \alpha^*$. 

To show \eqref{eqn:general_dpp}, observe that $\underline\alpha(\mu,\PiS,\KH)$ is the smallest maxmin control average-reward among the ones that appear in \eqref{eqn:general_dpp}. We will first show that 
$\underline\alpha(\mu,\PiS,\KH) \geq \alpha^*$. Then, we show $\overline\alpha(\mu,\PiS,\KS)\leq \alpha^*$ as well as $\overline\alpha(\mu,\PiH,\KH)\leq\alpha^*$. Combining these, we can conclude \eqref{eqn:general_dpp} and hence Theorem \ref{thm:general_dpp}. 

\textbf{Step 1: Show $\underline\alpha(\mu,\PiS,\KH) \geq \alpha^*$.}

Since $(u^*,\alpha^*)$ solves \eqref{eqn:r_bellman_eqn_const_gain}, for each $\epsilon > 0$, there exists a controller decision rule $\Delta:S\ra\cP(\cA)$ so that 
$$\inf_{p_s\in\cP_s} E_{\Delta(\cd|s),p_s}[r(s,A_0) - \alpha^* + u^*(X_1)]\geq u^*(s)-\epsilon.$$ 
Therefore, for any history-dependent adversarial policy $\kappa = (\kappa_0,\kappa_1,\ds)$ and any $s\in S,g_{k-1}\in\bd{G}_{k-1},k\geq 0$, we have that 
\begin{equation}\label{eqn:to_use_property_of_Delta}
    \sum_{a\in A}\Delta(a|s)\crbk{r(s,a) -\alpha^* + \sum_{s'\in S}\kappa_{k}(s'|g_{k-1},s,a) u^*(s')} \geq   u^*(s) -\epsilon.
\end{equation}
Using \eqref{eqn:to_use_property_of_Delta}, we have that
\begin{equation}\label{eqn:to_use_reward_sum_Delta}\begin{aligned}
&\sum_{k=0}^{n-1}\sum_{a\in A}\Delta(a|X_k)[r(X_k,a) - \alpha^*]\\
&\geq  -n\epsilon + \sum_{k=0}^{n-1}\sqbk{u^*(X_k)- \sum_{(a,s')\in A\times S}\kappa_{k}(s'|H_{k},a)\Delta(a|X_k)u^*(s') }\\
&= - n \epsilon + M_{u^*,n}^{\Delta,\kappa} - u^*(X_n) + u^*(X_0)
\end{aligned}
\end{equation}

On the other hand, notice that 
\begin{equation}\label{eqn:to_use_417}
 E_\mu^{\Delta,\kappa}[r(X_k,A_k)-\alpha^*] = E_\mu^{\Delta,\kappa}\sum_{a\in A}\Delta(a|X_k)[r(X_k,a)-\alpha^*].   
\end{equation}
Therefore, by \eqref{eqn:to_use_reward_sum_Delta}, $$\begin{aligned}
\frac{1}{n} \sum_{k=0}^{n-1}E^{\Delta,\kappa}_\mu[r(X_k,A_k)-\alpha^*] &= \frac{1}{n}E^{\Delta,\kappa}_\mu\sum_{k=0}^{n-1}\sum_{a\in A}\Delta(a|X_k)[r(X_k,a) - \alpha^*] \\
&\geq -\epsilon +  E^{\Delta,\kappa}_\mu \left[ \frac{M_{u^*,n}^{\Delta,\kappa}}{n} \right] + E^{\Delta,\kappa}_\mu\left[ \frac{u^*(X_0)-u^*(X_n)}{n} \right]\\
&\ra -\epsilon   
\end{aligned}$$
as $n\ra\infty$. Here, we use Proposition \ref{prop:mg} to conclude that $E^{\Delta,\kappa}_\mu M_{u^*,n}^{\Delta,\kappa} = 0$

So, we have that for arbitrary $\kappa\in\KH$, $$\liminf_{n\ra\infty}E_\mu^{\Delta,\kappa} \frac{1}{n}\sum_{k=0}^{n-1}r(X_k,A_k) \geq \alpha^*-\epsilon.$$
This implies that \begin{equation}\label{eqn:to_use_eps_opt_policy_bd}\inf_{\kappa\in\KH}\liminf_{n\ra\infty}E_\mu^{\Delta,\kappa} \frac{1}{n}\sum_{k=0}^{n-1}r(X_k,A_k) \geq \alpha^*-\epsilon.
\end{equation}
Moreover, since $\Delta\in \PiS$, we have that 
\begin{equation}\label{eqn:to_use_405}
\underline\alpha(\mu,\PiS,\KH)  =  \sup_{\pi\in\PiS}\inf_{\kappa\in\KH}\underline\alpha (\mu,\pi,\kappa) \geq \alpha^* - \epsilon.
\end{equation}
Since $\epsilon>0$ is arbitrary, we conclude that $\underline\alpha(\mu,\PiS,\KH)\geq \alpha^*$. 

\textbf{Step 2: Show $\overline\alpha(\mu,\PiS,\KS)\leq \alpha^*$ and $\overline\alpha(\mu,\PiH,\KH)\leq\alpha^*$. }

We consider an arbitrary history-dependent policy $\pi = (\pi_0,\pi_1,\ds)\in\PiH$.  Since $(u^*,\alpha^*)$ solves \eqref{eqn:r_bellman_eqn_const_gain}, for any $s\in S$, $g_{k-1}\in\bd{G}_{k-1}$ and $k\geq 0$,
$$\inf_{p_s\in\cP_s} E_{\pi_k(\cd|g_{k-1},s),p_s}[r(s,A_0) - \alpha^* + u^*(X_1)]\leq u^*(s).$$ 

Hence, there exists $\kappa_{k}(\cd |g_{k-1},s,\cd) \in \cP_{s}$ so that
\begin{equation}\label{eqn:to_use_438}
\sum_{a\in A}\pi_k(a|g_{k-1},s)\crbk{r(s,a) -\alpha^* +\sum_{s'\in S} \kappa_k(s'|g_{k-1},s,a)u^*(s')}\leq u^*(s) + \epsilon
\end{equation}
for each $s\in S$, $g_{k-1}\in\bd{G}_{k-1}$ and $k\geq 0$. 

Moreover, by the same argument, if $\pi_{k}(\cd|g_{k-1},s) = \pi(\cd|s)$ is Markov time-homogeneous, there exists $\kappa_{k}(\cd|g_{k-1},s,\cd) = \kappa(\cd|s,\cd)$ Markov time-homogeneous so that \eqref{eqn:to_use_438} holds. So, we have constructed an adversarial policy $\kappa := (\kappa_0,\kappa_1,\ds)\in \KH$ (or $\KS$) with $\set{\kappa_k:k\geq 0}$ specified above.

Therefore, with \eqref{eqn:to_use_438}, we have that
\begin{equation}\label{eqn:to_use_447}\begin{aligned}
&\sum_{k=0}^{n-1}\sum_{a\in A}\pi_k(a|H_k)[r(X_k,a) - \alpha^*]\\
&\leq  n\epsilon + \sum_{k=0}^{n-1}\sqbk{u^*(X_k)- \sum_{(a,s')\in A\times S}\pi_k(a|H_k)\kappa_{k}(s'|H_{k},a)u^*(s') }\\
&= n \epsilon + M_{u^*,n}^{\pi,\kappa} - u^*(X_n) + u^*(X_0)
\end{aligned}
\end{equation}

As in \eqref{eqn:to_use_417}, notice that $$E_\mu^{\pi,\kappa}[r(X_k,A_k)-\alpha^*] = E_\mu^{\pi,\kappa}\sum_{a\in A}\pi_k(a|H_k)[r(X_k,a)-\alpha^*].$$

Therefore, by \eqref{eqn:to_use_447}, $$\begin{aligned}
\frac{1}{n} \sum_{k=0}^{n-1}E^{\pi,\kappa}_\mu[r(X_k,A_k)-\alpha^*] &= \frac{1}{n}E^{\pi,\kappa}_\mu\sum_{k=0}^{n-1}\sum_{a\in A}\pi_k(a|H_k)[r(X_k,a) - \alpha^*] \\
&\leq \epsilon +  E^{\pi,\kappa}_\mu \left[\frac{M_{u^*,n}^{\pi,\kappa}}{n} \right]+ E^{\pi,\kappa}_\mu \frac{u^*(X_0)-u^*(X_n)}{n}\\
&\ra \epsilon   
\end{aligned}$$
as $n\ra\infty$. Here, we also use Proposition \ref{prop:mg} to conclude that $E^{\pi,\kappa}_\mu M_{u^*,n}^{\pi,\kappa} = 0$

So, we have that that for arbitrary $\pi\in\PiH$ (or $\pi\in\PiS$) there exists $\kappa\in\KH$ (or $\kappa\in\KS$) s.t. $$\limsup_{n\ra\infty}E_\mu^{\pi,\kappa} \frac{1}{n}\sum_{k=0}^{n-1}r(X_k,A_k) \leq \alpha^*+\epsilon.$$
Hence, 
$$\inf_{\kappa\in\KH}\limsup_{n\ra\infty}E_\mu^{\pi,\kappa} \frac{1}{n}\sum_{k=0}^{n-1}r(X_k,A_k) \leq \alpha^*+\epsilon$$
where $\kappa\in\KH$ is replaced by $\kappa\in\KS$ if $\pi\in\PiS$. 

Moreover, since $\pi\in \PiH$ (or $\pi\in \PiS$) can be any policy, we have that 
\begin{equation}\label{eqn:to_use_582}
\overline\alpha(\mu,\PiH,\KH)  =  \sup_{\pi\in\PiH}\inf_{\kappa\in\KH}\overline\alpha (\mu,\pi,\kappa) \leq \alpha^* + \epsilon.
\end{equation}
Since $\epsilon>0$ can be arbitrarily small, we conclude that $\overline\alpha(\mu,\PiH,\KH)\leq \alpha^*$. The same proof still holds when $\pi\in \PiS$ and $\kappa\in\KS$, leading to $\overline\alpha(\mu,\PiS,\KS)\leq \alpha^*$.

\textbf{Step 3: Combining Steps 1 and 2}
We combine the results from steps 1 and 2 to conclude that
$$\alpha^*\leq \underline\alpha(\mu,\PiS,\KH)\leq \underline\alpha(\mu,\PiH,\KH)\leq \overline\alpha(\mu,\PiH,\KH)\leq \alpha^*,$$
$$\alpha^*\leq \underline\alpha(\mu,\PiS,\KH)\leq \overline\alpha(\mu,\PiS,\KH)\leq \overline\alpha(\mu,\PiH,\KH)\leq \alpha^*,$$
$$\alpha^*\leq \underline\alpha(\mu,\PiS,\KH)\leq \underline\alpha(\mu,\PiS,\KS)\leq \overline\alpha(\mu,\PiS,\KS)\leq \alpha^*,$$
$$\alpha^*\leq \underline\alpha(\mu,\PiS,\KH)\leq \overline\alpha(\mu,\PiS,\KH)\leq \overline\alpha(\mu,\PiS,\KS)\leq \alpha^*.$$
These inequalities imply \eqref{eqn:general_dpp}. 
\end{proof}

\subsection{Proof of Theorem \ref{thm:interchange_imply_dpp}}
\begin{proof}

Since $(u^*,\alpha^*)$ solves \eqref{eqn:inf_sup_eqn_const_gain}, we have that there exists $\psi:S\times A\ra\cP(\cS)$ s.t. $\psi(\cd|s,\cd)\in\cP_s$ for all $s\in S$ and 
$$\sup_{\phi\in\cQ} E_{\phi,\psi(\cd|s,\cd)}[r(s,A_0) - \alpha^* + u^*(X_1)]\leq u^*(s)+\epsilon.$$ 
Thus, for any history-dependent policy $\pi = (\pi_0,\pi_1,\ds)$, and $s\in S,g_{k-1}\in\bd{G}_{k-1},k\geq 0$, 
\begin{equation}\label{eqn:to_use_411}\sum_{a\in A}\pi_{k}(a|g_{k-1},s)\sqbk{r(s,a) -\alpha^* + \sum_{s'\in S}\psi(s'|s,a)u^*(s')} \leq  u^*(s)+\epsilon.
\end{equation}

Note that
$$\begin{aligned}
\frac{1}{n}\sum_{k=0}^{n-1}E^{\pi,\psi}_\mu [r(X_k,A_k)-\alpha^*]&= \frac{1}{n}\sum_{k=0}^{n-1}E^{\pi,\psi}_\mu E^{\pi,\psi}_\mu [r(X_k,A_k)-\alpha^*|\cH_k] \\
&=\frac{1}{n}\sum_{k=0}^{n-1}E^{\pi,\psi}_\mu \sqbk{\sum_a\pi_{k}(a|G_{k-1},X_k)r(X_k,a) -\alpha^*}\\
&\stackrel{(i)}{\leq} \frac{1}{n}\sum_{k=0}^{n-1}\crbk{E^{\pi,\psi}_\mu \sqbk{u^*(X_k) +\epsilon} - E^{\pi,\psi}_\mu \sqbk{\sum_{a\in A,s'\in S}\pi_k(a|H_{k})\psi(s'|X_k,a)u^*(s')}} \\
&= \epsilon + \frac{1}{n}E^{\pi,\psi}_\mu M_{u^*,n}^{\pi,\psi} + E^{\pi,\psi}_\mu \frac{u^*(X_0)- u^*(X_n)}{n}\\
&= \epsilon  + E^{\pi,\psi}_\mu \frac{u^*(X_0)- u^*(X_n)}{n}\\
&\ra \epsilon
\end{aligned}$$
$n\ra\infty$, where $(i)$ follows from \eqref{eqn:to_use_411}.

So, we have that
$$\limsup_{n\ra\infty}E_\mu^{\pi,\psi} \frac{1}{n}\sum_{k=0}^{n-1}r(X_k,A_k) \leq \alpha^* +\epsilon.$$
Since $\pi\in\PiH$ is arbitrary, we can conclude that
\begin{equation}\label{eqn:to_use_430}
\inf_{\kappa\in\KS}\sup_{\pi\in\PiH}\overline\alpha(\mu,\pi,\kappa)\leq \sup_{\pi\in\PiH}\overline\alpha(\mu,\pi,\psi)\leq \alpha^* +\epsilon.
\end{equation}

On the other hand, since $(u^*,\alpha^*)$ solves \eqref{eqn:r_bellman_eqn_const_gain}, Theorem \ref{thm:general_dpp} and the proofs in Appendix \ref{section:proof:thm:general_dpp} are still valid. In particular, by \eqref{eqn:to_use_405}, still holds. Therefore, combining \eqref{eqn:to_use_405} with \eqref{eqn:to_use_430}, we have that 
\[\begin{aligned}\alpha^*-\epsilon &\leq\sup_{\pi\in\PiS}\inf_{\kappa\in\KH} \underline\alpha(\mu,\pi,\kappa)\leq \sup_{\pi\in\PiH}\inf_{\kappa\in\KH} \underline\alpha(\mu,\pi,\kappa)\\
&\leq \inf_{\kappa\in\KH}\sup_{\pi\in\PiH} \underline\alpha(\mu,\pi,\kappa)\leq\inf_{\kappa\in\KS}\sup_{\pi\in\PiH}\overline\alpha(\mu,\pi,\kappa) \\
&\leq\alpha^* +\epsilon.
\end{aligned}\]
Since $\epsilon$ is arbitrary, we have that 
$$\alpha^* = \sup_{\pi\in\PiS}\inf_{\kappa\in\KH} \underline\alpha(\mu,\pi,\kappa) =\inf_{\kappa\in\KS}\sup_{\pi\in\PiH}\overline\alpha(\mu,\pi,\kappa).$$
This implies the statement of the theorem, as $\sup_{\pi\in\PiS}\inf_{\kappa\in\KH} \underline\alpha(\mu,\pi,\kappa)$ is the smallest and \linebreak
$\inf_{\kappa\in\KS}\sup_{\pi\in\PiH}\overline\alpha(\mu,\pi,\kappa)$ is the largest among all the relevant values in the statement of the theorem. 
\end{proof}

\subsection{Proof of Theorem \ref{thm:policy_opt_gap}}
\begin{proof}
    We observe that \textbf{Step 1} in the proof of Theorem \ref{thm:general_dpp} shows that the policy $\Delta$ assumed in the statement of  Theorem \ref{thm:policy_opt_gap} satisfies \eqref{eqn:to_use_eps_opt_policy_bd}; i.e. 
    $$\inf_{\kappa\in\KH}\liminf_{n\ra\infty}E_\mu^{\Delta,\kappa} \frac{1}{n}\sum_{k=0}^{n-1}r(X_k,A_k) \geq \alpha^*-\epsilon.$$

By Theorem \ref{thm:general_dpp} and set inclusion arguments, we derive that
$$\alpha^* = \overline {\alpha}(\mu,\PiS,\KH) =\overline{\alpha}(\mu,\PiS,\KS) = \underline{\alpha}(\mu,\PiS,\KS) \geq \inf_{\kappa\in\KS}\underline{\alpha}(\mu,\Delta,\kappa)\geq \inf_{\kappa\in\KH}\underline{\alpha}(\mu,\Delta,\kappa)\geq \alpha^*-\epsilon.$$ Rearranging these inequalities gives Theorem \ref{thm:policy_opt_gap}. 
    
\end{proof}

\section{Proof of Theorem \ref{thm:existence_of_sol}}
\begin{proof}
\ref{enum:discounted_bd_span}$\implies$\ref{enum:arbe_exist_sol}: We fix a reference state $s_0\in S$ and define $$u_\gamma = v_\gamma^* - v_\gamma^*(s_0),\qquad 
\alpha_\gamma = (1-\gamma)v_\gamma^*(s_0).$$

Since $v_\gamma^*$ solves \eqref{eqn:discounted_r_bellman_eqn}, we observe that 
\begin{equation}\label{eqn:to_use_444_u_eqn}
    \begin{aligned}v_\gamma^*(s) - v_\gamma^*(s_0) &= \sup_{\phi\in \cQ}\inf_{p_s\in\cP_s}E_{\phi,p_s}[r(s,A_0) +\gamma (v_\gamma^* (X_1)-v_\gamma^*(s_0))] - (1-\gamma)v_\gamma^*(s_0)\\
    u_\gamma(s) &= \sup_{\phi\in \cQ}\inf_{p_s\in\cP_s}E_{\phi,p_s}[r(s,A_0)- \alpha_\gamma +\gamma u_\gamma(X_1)] 
\end{aligned}
\end{equation}

From \citet[Proposition 1]{wang2023foundationRMDP}, $\norminf{v_\gamma^*}\leq 1/(1-\gamma)$. So, $0\leq \alpha_\gamma\leq 1$. Moreover, by  \ref{enum:discounted_bd_span}, $\norminf{u_\gamma}\leq \spnorm{v_\gamma^*} \leq C <\infty$ uniformly in $\gamma$. Hence $(u_\gamma ,\alpha_\gamma)\in [-C,C]^{|S|}\times [0,1]$ for all $\gamma$. As $ [-C,C]^{|S|}\times [0,1]$ is compact in the sup metric, there exists a convergent subsequence $\set{(u_{\gamma_n},\alpha_{\gamma_n}):n=1,2,\ds}$ with $(u_*,\alpha_*) := \lim_{n\ra \infty}(u_{\gamma_n},\alpha_{\gamma_n})$.

Next, we would like to take the limit $n\ra\infty$ on both sides of \eqref{eqn:to_use_444_u_eqn}, with $\gamma$ replaced by $\gamma_n$. To do this, we define for $\gamma\in[0,1]$, $(u,\alpha)\in [-C,C]^{|S|}\times [0,1]$, $\phi\in\cP(\cA)$, $p_s\in\set{A\ra \cP(\cS)}$,  $$f_s(\gamma,u,\alpha,\phi,p_s) = E_{\phi,p_s}[r(s,A_0)- \alpha +\gamma u(X_1)],$$ which is a continuous function. 

We first note that since $\cP_s$ is bounded and the mapping $p_s\ra f_s(\gamma,u,\alpha,\phi,p_s) $ is continuous, 
$$\inf_{p_s\in\cP_s}f_s(\gamma,u,\alpha,\phi,p_s) = \min_{p_s\in\overline \cP_s}f_s(\gamma,u,\alpha,\phi,p_s)$$ 
where $\overline \cP_s$ is the closure of $\cP_s$. 

Since $\overline \cP_s$ is compact and does not depend on $\gamma,\phi,u,\alpha$, by Berge’s maximum theorem \citep[VI.3, Theorem 1 \& 2]{berge1963topological}, the mapping $$(\gamma,u,\alpha,\phi)\ra m_s(\gamma,u,\alpha,\phi):=\min_{p_s\in\overline{\cP}_s}f_s(\gamma,u,\alpha,\phi,p_s)$$ is continuous for $\gamma\in[0,1]$, $(u,\alpha)\in [-C,C]^{|S|}\times [0,1]$, and $\phi\in\cP(\cA)$. 

Apply the same argument by defining $\overline \cQ$ as  the closure of $\cQ$, we have that $$M_s(\gamma,u,\alpha) = \max_{\phi\in\overline{\cQ}} m_s(\gamma,u,\alpha,\phi) = \sup_{\phi\in \cQ} m_s(\gamma,u,\alpha,\phi) = \sup_{\phi\in \cQ}\inf_{p_s\in\cP_s}f_s(\gamma,u,\alpha,\phi,p_s)$$ is continuous for $\gamma\in[0,1]$ and $(u,\alpha)\in [-C,C]^{|S|}\times [0,1]$. 

Therefore, we have that $$\lim_{n\ra\infty}M_s(\gamma_n,u_{\gamma_n},\alpha_{\gamma_n}) = M_s(1,u_*,\alpha_*) =  \sup_{\phi\in \cQ}\inf_{p_s\in\cP_s}E_{\phi,p_s}[r(s,A_0)- \alpha_* + u_*(X_1)].$$
This and \eqref{eqn:to_use_444_u_eqn} implies that 
$$ u_*(s)= \sup_{\phi\in \cQ}\inf_{p_s\in\cP_s}E_{\phi,p_s}[r(s,A_0)- \alpha_* + u_*(X_1)]$$
i.e. $(u_*,\alpha_*)$ solves \eqref{eqn:r_bellman_eqn_const_gain}.

\ref{enum:arbe_exist_sol}$\implies$\ref{enum:discounted_bd_span}: Let $(u^*,\alpha^*)$ be a solution to \eqref{eqn:r_bellman_eqn_const_gain}. Due to the solutions of \eqref{eqn:r_bellman_eqn_const_gain} being shift-invariant, w.l.o.g., we assume that $u^*\geq 0$ and $\min_{s\in S}u^*(s) = 0$. To simplify notation, we define the discounted Bellman operator $$\cT_{\gamma}[v] :=\sup_{\phi\in\cQ}\inf_{p_s\in\cP_s}E_{\phi,p_s}[r(s,A_0) + \gamma v(X_1)]. $$
Then $\cT_\gamma[v_\gamma^*] = v_\gamma^*$, where $v_\gamma^*$ is the unique fixed-point. 

We define two auxiliary values 
\begin{equation}\label{eqn:to_use_def_over_under_v_gamma}\overline {v}_\gamma:= \frac{\alpha^*}{1-\gamma} + u^*,\quad \underline{v}_\gamma:=\frac{\alpha^*}{1-\gamma} + u^* - \spnorm{u^*}.
\end{equation}

\textbf{Step 1: }We show that $\cT_{\gamma}[\overline {v}_\gamma]\leq \overline{v}_\gamma$ and $\cT_{\gamma}[\underline {v}_\gamma]\geq \underline{v}_\gamma$.

We observe that for all $s\in S$, 
$$\begin{aligned}
    \cT_{\gamma}[\overline {v}_{\gamma}](s)&= \sup_{\phi\in\cQ}\inf_{p_s\in\cP_s}E_{\phi,p_s}[r(s,A_0) + \gamma \overline {v}_{\gamma}(X_1)]\\
    &= \frac{\gamma\alpha^*}{1-\gamma} + \alpha^*+ \sup_{\phi\in\cQ}\inf_{p_s\in\cP_s}E_{\phi,p_s}[r(s,A_0) -\alpha^*+ \gamma u^*(X_1)]\\
    &=  \frac{\alpha^*}{1-\gamma} + \sup_{\phi\in\cQ}\inf_{p_s\in\cP_s}E_{\phi,p_s}[r(s,A_0) -\alpha^*+  u^*(X_1) - (1-\gamma)u^*(X_1)]\\
    &\stackrel{(i)}{\leq}  \frac{\alpha^*}{1-\gamma} + \sup_{\phi\in\cQ}\inf_{p_s\in\cP_s}E_{\phi,p_s}[r(s,A_0) -\alpha^*+  u^*(X_1)]\\
    &= \overline{v}_\gamma (s)
\end{aligned}$$
where $(i)$ follows from the choice that $u^*\geq 0$ and the last equality uses the assumption that $(u^*,\alpha^*)$ solves \eqref{eqn:r_bellman_eqn_const_gain}. On the other hand, 

$$\begin{aligned}
    \cT_\gamma[\underline{v}_\gamma](s)
    &= \frac{\alpha^*}{1-\gamma} + \sup_{\phi\in\cQ}\inf_{p_s\in\cP_s}E_{\phi,p_s}[r(s,A_0) -\alpha^*+ \gamma (u^*(X_1) - \spnorm{u^*})]\\
    &= \frac{\alpha^*}{1-\gamma} -  \spnorm{u^*} + \sup_{\phi\in\cQ}\inf_{p_s\in\cP_s}E_{\phi,p_s}[r(s,A_0) -\alpha^*+ u^*(X_1) +  (1-\gamma)(\spnorm{u^*}-u^*(X_1))]\\
    &\stackrel{(i)}{\geq}\frac{\alpha^*}{1-\gamma} -  \spnorm{u^*} + \sup_{\phi\in\cQ}\inf_{p_s\in\cP_s}E_{\phi,p_s}[r(s,A_0) -\alpha^*+ u^*(X_1) ]\\
    &=\underline{v}_\gamma (s)
\end{aligned}$$
where $(i)$ follows from $u^*\geq 0$ and $\min_{s\in S}u^*(s) = 0$ hence $\spnorm{u^*}-u^*= \norminf{u^*}-u^* \geq 0 $. 

\textbf{Step 2: }We prove that $v_{\gamma}^*$, the solution to \eqref{eqn:discounted_r_bellman_eqn}, is upper bounded by $\overline{v}_\gamma$ and lower bounded by $\underline{v}_\gamma$; i.e. $$\underline{v}_\gamma\leq v_{\gamma}^*\leq \overline{v}_\gamma. $$

To achieve this, we will use the fact that $\cT_\gamma$ is a monotone $\gamma$-contraction. 

First, the contraction property of $\cT_\gamma$ is well known (see \citet{wang2023foundationRMDP}). We then show that $\cT_\gamma$ is a monotone operator; i.e., $\cT_\gamma [u]\geq \cT_\gamma [v]$ if $u\geq v$. This is straightforward
$$\cT_{\gamma}[u](s)= \sup_{\phi\in\cQ}\inf_{p_s\in\cP_s}E_{\phi,p_s}[r(s,A_0) + \gamma [u(X_1)-v(X_1)]+\gamma v(X_1)]\geq \cT_{\gamma}[v](s)$$
where we used that $u(X_1)-v(X_1)\geq 0$.

Next, we check by induction that $$\cT_\gamma^k[\overline{v}_\gamma] :=\underbrace{(\cT_{\gamma}\circ \ds \circ \cT_{\gamma})}_{\times k} [\overline v]\leq \overline{v}$$ for all $k\geq 1$. The base case $k = 1$ follows from the previous proof. For the induction step, assume that $\cT_{\gamma}^k[\overline{v}_{\gamma}] \leq \overline{v}_{\gamma}$. By the monotonicity of $\cT_{\gamma}$, we have that $$\cT_{\gamma}^{k+1}[\overline{v}_{\gamma}] = \cT_{\gamma}[\cT_{\gamma}^k[\overline{v}_{\gamma}]]\leq \cT_{\gamma}[\overline{v}_{\gamma}]\leq \overline{v}_{\gamma},$$ completing the induction step. 

On the other hand, by the contraction property, $\overline v_\gamma \geq \cT_{\gamma}^k[\overline{v}_\gamma]\ra v_{\gamma}^*$ as $k\ra\infty$. So, we have that $\overline{v}_\gamma\geq v_{\gamma}^*$.

Similarly, we show that $v_{\gamma}^*$ is lower bounded by $\underline{v}_\gamma$. Again, we apply the same induction argument. We see that the base case holds due to $\cT_{\gamma}[\underline {v}_\gamma]\geq \underline{v}_\gamma$, and the induction step follows from the monotonicity of $\cT_{\gamma}$. Therefore, by the contraction property, $\underline{v}_\gamma\leq \cT_{\gamma}^k[\underline{v}_\gamma]\ra v_{\gamma}^*$ as $k\ra\infty$. So, we have that $\underline{v}_\gamma\leq v_{\gamma}^*$. 

\textbf{Step 3:} We conclude the proof by bounding the span of $v_\gamma^*$.

Since $\underline{v}_\gamma\leq v_{\gamma}^*\leq \overline{v}_\gamma $, $$\spnorm{v_\gamma^*} = \max_{s\in S}v_\gamma^*(s) - \min_{s\in S}v_\gamma^*(s) \leq \max_{s\in S}\overline {v}_\gamma(s) - \min_{s\in S}\underline{v}_\gamma(s)= 2\spnorm{u^*}$$
where the last equality follows from the definition of $\overline{v}_\gamma$ and $\underline{v}_\gamma$ in \eqref{eqn:to_use_def_over_under_v_gamma}. 
\end{proof}

\section{Proof of Auxiliary Lemmas in Section \ref{sec:proof_of_thm_wc_ctrl}}\label{sec:proof_of_aux_lemma_wc_ctrl}
\subsection{Proof of Lemma \ref{lemma:wc_ctrl_implication}}
\begin{proof}
Fix $\Delta$ and its communicating set $C_\Delta\subseteq S$; write $C_\Delta^c:=S\setminus C_\Delta$. We separately consider $w\in C_\Delta$ and $w\in C_\Delta^c$. 

\emph{Case 1: $w\in C_\Delta$.}
Since $w\in C_\Delta$, by weak communication, there exist $p = p^{w,y,\Delta}\in\cP$ and $N_1\le |C_\Delta|$ such that
$p_\Delta^{N_1}(y|w)>0$. This proves the claim.

\emph{Case 2: $w\in C_\Delta^c$.}
We choose an arbitrary $p\in\cP$. By weak communication, there exists a non-repeating path
$w=s_0\ra s_1\ra\cdots\ra s_k=x\in C_\Delta$ with $k\le |C_\Delta^c|$ and
$\prod_{j=0}^{k-1} p_\Delta(s_{j+1}|s_j)>0$. Define
\[
q(\cdot|s,\cdot):=p(\cdot|s,\cdot)\in\cP_s,\quad \forall s\in C_\Delta^c.
\]
Inside $C_\Delta$, we apply the $p^{x,y,\Delta}\in\cP$ and $N_1\le |C_\Delta|$ in Case 1, with
$p_\Delta^{N_1}(y|x)>0$, and set
\[
q(\cdot|s,\cdot):=p^{x,y,\Delta}(\cdot|s,\cdot)\in\cP_s,\quad \forall  s\in C_\Delta.
\]

By construction, under $q_\Delta$, there is a non-repeating path $w=s_0\ra\cdots\ra s_k=x= c_0\ra \ds \ra c_{N_1} = y$ such that
$\prod_{j=0}^{k-1} q_\Delta(s_{j+1}|s_j)>0$ and $k\le |C_\Delta^c|$. Hence, $q_\Delta^{N}(y|w)>0$ with $N:=k+N_1 \le |C_\Delta^c|+|C_\Delta|=|S|$. Moreover, By S-rectangularity, $q\in\cP$. This proves the claim with $p =q$. 
\end{proof}

\subsection{Proof of Lemma \ref{lemma:construct_K_Delta}}

\begin{proof}


By weak communication, for each $p\in\cP$, every state in $C_\Delta^c$ is transient under controller $\Delta$. Thus $M_{\Delta,p}^\Delta$ is the transient block of $p_{\Delta}$ in the canonical classification. So, letting $T_{C_{\Delta}}:=\inf\set{t\geq 0: X_t\in C_{\Delta}}$, a first transition analysis argument suggests that $I - M_{\Delta,p}^\Delta$ is invertible and for any $z\in C_{\Delta}^c$, 
$$[(I - M_{\Delta,p}^\Delta)^{-1} e ](z) = E^{\Delta,p}_z T_{C_\Delta} < \infty.$$

Consider the mapping
\[
g:(\eta,p)\ra \det\crbk{I - M_{\eta,p}^\Delta}.
\]
The mapping $(\eta,p)\ra M_{\eta,p}^\Delta$ is continuous (entrywise), and the determinant is a polynomial in entries, hence $g$ is continuous on $\{S\to\cQ\}\times\cP$.

Since $\cP$ is compact and does not depend on $\eta$, by Berge’s maximum theorem \citep[VI.3, Theorem 1 \& 2]{berge1963topological}
\[
h:\eta\ra \min_{p\in\cP}|g(\eta,p)|
\]
is also continuous in $\eta$. 

On the other hand, at $\eta=\Delta$,  one has that for every $p$, $g(\Delta,p)\neq 0$, by the invertibility of $I - M_{\Delta,p}^\Delta$.  So, for some $p'\in\cP$, $$h(\Delta) = \min_{p\in\cP}|g(\eta,p)| = \abs{\det\crbk{I - M_{\Delta,p'}^\Delta}} > 0.$$ Therefore, by the continuity of $h$  there exists an open neighborhood $K'_{\Delta}$ s.t. for any $\eta\in K'_{\Delta}$, $M_{\eta,p}^\Delta$ is not all 0 for any $p$, and $$ h(\eta) = \min_{p\in \cP}\abs{  \det\crbk{I - M_{\eta,p}^\Delta} }> 0. $$ Therefore, we conclude that $\forall (\eta,p) \in K'_{\Delta} \times \cP$, $I-M_{\eta,p}^\Delta$ is invertible.

With this, we define
\[
\phi(\eta,p):= e^\top(I - M_{\eta,p}^\Delta)^{-1} e
\]
on $K'_\Delta\times\cP$, which is also continuous by the continuity of matrix inversion. So, applying Berge’s maximum theorem again, we conclude that by the compactness of $\cP$,
\[
\eta \ra \max_{p\in\cP} \phi(\eta,p),\quad \eta\in K'_\Delta.
\]
is continuous. Moreover, note that by the Neumann series representation, for all $(\eta,p)\in K'_\Delta\times\cP$
$$e^\top(I - M_{\eta,p}^\Delta)^{-1} e = \sum_{n=0}^{\infty} e^\top (M_{\eta,p}^\Delta)^n e \geq 0$$

Finally, to conclude the lemma, we note that $$0\leq \max_{p\in\cP} \phi(\Delta,p) = \max_{p\in\cP}e^\top (I - M_{\Delta,p}^\Delta)^{-1} e < \infty,$$ and by continuity, there exists an open neighborhood $K_{\Delta}\subseteq K'_{\Delta}$ such that for all $\eta\in \overline{K_{\Delta}}$, 
$$0\leq \max_{p\in\cP}e^\top(I - M_{\eta,p}^\Delta)^{-1} e = \max_{p\in\cP} \phi(\eta,p) \leq  \max_{p\in\cP} \phi(\Delta,p)  + 1 = \max_{p\in\cP}e^\top (I - M_{\Delta,p}^\Delta)^{-1} e + 1 < \infty.$$
This concludes the proof of Lemma \ref{lemma:construct_K_Delta}. 
\end{proof}

\subsection{Proof of Lemma \ref{lemma:wc_bd_E_hitting_time}}

\begin{proof}

Consider any $w\in S$. By Lemma \ref{lemma:wc_ctrl_exists_adv_lower_bd_prob}, for any $\Delta:S\ra\cQ$ with $\Delta\in G_{\Delta_k}$ and $y\in C_{\Delta_k}$, there is $p\in\cP$ and $N\leq |S|$ (both can be dependent on $w,y,\Delta$) such that $$P^{\Delta,p}_w(\tau_y \leq |S|)\ge p_\Delta^N(y|w)\geq \delta.$$ 

We will first show that for some $\delta'>0$ independent of $\Delta$ and $y$, there is $q\in\cP$ s.t. 
\begin{equation}
 \min_{w\in S}P_w^{\Delta,q}(\tau_y\leq |S|)\geq \delta';\label{eqn:to_use_unif_hitting_prob_lb}
\end{equation} i.e. choice $p\in\cP$ in Lemma \ref{lemma:wc_ctrl_exists_adv_lower_bd_prob} can be made independent of $w$.

To this end, let us define $q = q^{y,\Delta}\in\cP$ algorithmically as follows. We will iteratively assign $q(\cd|s,\cd)\in\cP_s$ until all $\set{q(\cd|s,\cd): s\in S}$ has been assigned. 
    
We initialize the algorithm by assigning $q(\cd|y,\cd) = p_y $ for an arbitrary $p_y\in \cP_y$. Then, let $V = \set{y}$ be the assigned states, and $V^c = S\setminus V$ the complement in $S$ be the unassigned states.  
\begin{enumerate}
    \item Choose any unassigned state $s_0\in V^c$. Then by Lemma \ref{lemma:wc_ctrl_exists_adv_lower_bd_prob} there exists $p = p^{s_0,y,\Delta}\in\cP$ and $N$ s.t. $p_{\Delta}^N(y|s_0)\geq \delta$. Therefore, there exists a path $s_0\ra s_1\ra,\ds,\ra s_N = y$ s.t. $p_\Delta(s_{k+1}|s_k) > 0$.  

    Moreover, since there are at most $|S|^N$ paths from $s_0$ to $y$ in $N$ steps, there must be one path with probability at least $\delta|S|^{-N}$ under $p_\Delta$. Let $s_0\ra s_1\ra,\ds,\ra s_N = y$ be this path. 
    
    Note that, in general, this path could be repeating, i.e., $s_i = s_j$ for some $i <j$. However, we can ``trim off'' the in-between segment to get $s_0\ra \ds \ra s_i\ra s_{j+1}\ra  \ds \ra s_N = y$. This is again a path with probability at least $\delta|S|^{-N}$ under $p_\Delta$. We trim until obtaining a non-repeating path with probability at least $\delta|S|^{-N}$ and relabel it with $s_0\ra s_1\ra ,\ds ,\ra s_k = y$ for some $k\leq N$. 
    
    Therefore, we have that on this path, for all $i\leq k-1$,
    $$p_\Delta(s_{i+1}|s_i) \geq \prod_{i=0}^{k-1} p_\Delta(s_{i+1}|s_i) \geq \delta|S|^{-N}. $$  
    
    \item Let $j = \min\set{i\geq 1: s_i\in V}$ be the first index so that $s_j$ is assigned. So, $s_i\in V^c$ for all $i\leq j-1$. We assign $q(\cd|s_i,\cd):= p^{s_0,y,\Delta}(\cd|s_i,\cd)\in \cP_{s_i}$. Note that the construction of this path implies that $q_{\Delta}(s_{i+1}|s_i) \geq \delta|S|^{-N}$ for all $i\leq j-1$.

    Moreover, at the current iteration, $q_\Delta(\cd|s)$ is well-defined for all $s\in V$. Since $s_j\in V$, there is a non-repeating path $\set{s_j = s_j',s'_{j+1},\ds ,s'_{k} = y}\subseteq V$ s.t. $q_\Delta(s'_{i+1}|s'_{i}) \geq \delta|S|^{-N}$. 
    
    Therefore, after assigning $q(\cd|s_i,\cd)$ for $i\leq j-1$, we have a new path $s_0\ra \ds\ra s_j\ra s_{j+1}'\ra \ds \ra s_k' = y$ with positive one step transition probabilities at least $\delta|S|^{-N}$ under $q_\Delta$. We record this path that leads to $y$. 
    
    \item Update $V \la V\cup\set{s_0,\ds,s_{j-1}}$. 
\end{enumerate}
Iterate until $V = S$.

Note that the algorithm terminates in at most $|S|$ iterations, producing $q\in\cP$, as we always assign $q(\cd|s,\cd)\in\cP_s$. Moreover, it produces a directed graph whose edges correspond to a positive transition probability at least $\delta|S|^{-N}$ under $q_\Delta(\cd|\cd)$, ensuring that every state can reach $y$ in at most $|S|$ steps.

Therefore, we conclude that with $q = q^{y,\Delta}\in\cP$ constructed by the above algorithm, 
$$
\begin{aligned}\min_{w\in S}P_w^{\Delta,q}(\tau_y\leq |S|)\geq \crbk{\delta|S|^{-N}}^{|S|}\geq \crbk{\delta|S|^{-|S|}}^{|S|}=:\delta' > 0. 
\end{aligned}$$
Note that $\delta'$ is independent of $\Delta$ and $y$. This shows \eqref{eqn:to_use_unif_hitting_prob_lb}. 
    
Under a standard renewal-type argument, we turn the probability bound in \eqref{eqn:to_use_unif_hitting_prob_lb} into the expected hitting time bound in Lemma \ref{lemma:wc_bd_E_hitting_time}.

First, we show that 
\begin{equation}\label{eqn:to_use_renewal_bound_tauy}
   \max_{w\in S} P_w^{\Delta,q}(\tau_y > m|S|)\leq (1-\delta')^m. 
\end{equation}

We prove this by induction on $m$. The base case $m = 1$ follows directly from \eqref{eqn:to_use_unif_hitting_prob_lb} that \begin{equation}\label{eqn:to_use_base_case_1-delta}
\max_{w\in S} P_w^{\Delta,q}(\tau_y > |S|) = 1-\min_{w\in S}P_w^{\Delta,q}(\tau_y \leq |S|)\leq 1-\delta'.\end{equation}
For the induction step, note that for any $x$
\[\begin{aligned}
    P_w^{\Delta,q}(\tau_y > (k+1)|S|)&= E_w^{\Delta,q}\1\set{\tau_y > (k+1)|S|}\\
    &=E_w^{\Delta,q}E_w^{\Delta,q}\sqbkcond{\1\set{\tau_y > (k+1)|S|}}{\cH_{k|S|}}\\
    &\stackrel{(i)}{=} E_w^{\Delta,q}\sqbk{\1\set{\tau_y>k|S|}E_w^{\Delta,q}\sqbkcond{\1\set{\tau_y > (k+1)|S|}}{\cH_{k|S|}}}\\
    &\stackrel{(ii)}{=} E_w^{\Delta,q}\sqbk{\1\set{\tau_y>k|S|}E_{X_{k|S|}}^{\Delta,q}\sqbk{\1\set{\tau_y > |S|}}}\\
    &= E_w^{\Delta,q}\sqbk{\1\set{\tau_y>k|S|}P_{X_{k|S|}}^{\Delta,q}(\tau_y > |S|) }\\
    &\stackrel{(iii)}{\leq } E_w^{\Delta,q}[\1\set{\tau_y>k|S|}](1-\delta')\\
    &\leq (1-\delta')^{k+1}\\
\end{aligned}\]
where $(i)$ follows from $\tau_y$ is a $\cH_t$-stopping time with$\set{\tau_y >k|S|} = \set{\tau_y\leq k|S|}^c\in \cH_{k|S|}$, as well as $\1\set{\tau_y > (k+1)|S|} =\1\set{\tau_y > (k+1)|S|}\1\set{\tau_y > k|S|}$, $(ii)$ is due to the Markov property, and $(iii)$ follows from \eqref{eqn:to_use_base_case_1-delta}. This completes the induction step and shows \eqref{eqn:to_use_renewal_bound_tauy}.

We then prove Lemma \ref{lemma:wc_bd_E_hitting_time} using \eqref{eqn:to_use_renewal_bound_tauy}. Note that since $\tau_y$ is non-negative, for $w\in S$, $w\neq y$, 
$$\begin{aligned}
E_w^{\Delta,q}[\tau_y] &= \sum_{k\geq 0}P_w^{\Delta,q}(\tau_y\geq k) \\
&=\sum_{k\geq 0}P_w^{\Delta,q}(\tau_y> k) \\
&\leq |S| +\sum_{k\geq 1} |S|P_w^{\Delta,q}(\tau_y\geq k|S|)\\
&\leq |S|\sum_{k=0}^\infty (1-\delta')^k\\
&\leq \frac{|S|}{\delta'}.
\end{aligned}$$
Of course $E_y^{\Delta,q}[\tau_y] = 0\leq |S|/\delta'$. This implies Lemma \ref{lemma:wc_bd_E_hitting_time}. 
\end{proof}

\section{Proof of Theorem \ref{thm:wc_adv}}
\label{sec:proof_wc_adv}

The proof follows arguments very similar to those in Theorem \ref{thm:wc_ctrl}. To avoid excessive repetition, we focus on explaining how the earlier proof carries over and highlighting the necessary modifications.

As in the proof of Theorem \ref{thm:wc_ctrl}, we primarily address the weakly communicating case in assumption (1). Since a communicating adversary is a special case, the argument under assumption (2) carries over with only minor changes, the main difference being the lack of convexity of $\cQ$ relative to assumption (1).

\begin{proof}
From the argument proving Theorem \ref{thm:existence_of_sol}, it follows that if the solutions $\set{v'_\gamma:\gamma\in(0,1)}$ to \eqref{eqn:discounted_inf_sup_eqn} have uniformly bounded span, then \eqref{eqn:inf_sup_eqn_const_gain} has a solution. Therefore, we now proceed to establish the uniform boundedness of $\spnorm{v'_\gamma}$, mirroring the argument used in the proof of Theorem \ref{thm:wc_ctrl}.

\subsection*{Preliminary Constructions}
Mirroring the proof of Lemma \ref{lemma:wc_ctrl_implication}, it is easy to see that the following Lemma holds. 

\begin{lemma}\label{lemma:wc_adv_implication}
If the stationary adversary policy $p\in\cP$ is weakly communicating and $\cP$ is S-rectangular, then for any $w\in S$ and any $y\in C_p$, there exist $\Delta:S\to\cQ$ and $N\le |S|$ such that $p_\Delta^{N}(y|w)>0$.
\end{lemma}

Next, we will leverage the compactness of $\cP$ to construct a finite subset $B'$ of stationary adversary policies that yields a uniform lower bound on hitting probabilities. 

Assume the adversary is weakly communicating, $\cP$ is S-rectangular, and $\cQ$ and $\cP$ are compact. By Lemma \ref{lemma:wc_adv_implication}, for any stationary policy $p\in\cP$ and any $w\in S$, $y\in C_p$, there exist $\Delta$ and $N\le |S|$ (both possibly depending on $(w,y,p)$) such that $p_\Delta^{N}(y|w)>0$.

Note that if we fix $\Delta = \Delta^{w,y,p}$ and $N = N^{w,y,p}$ given by Lemma \ref{lemma:wc_adv_implication}, then the mapping $q\ra q_\Delta^N(y|w)$ is continuous in $q$. So, there must exist an open neighborhood $O^{w,y}_p\subseteq \cP$ of $p$ such that
\begin{equation}\label{eqn:wc_adv_prob_bound}
    q_\Delta^N(y|w) \geq \tfrac12 p_\Delta^N(y|w), \quad \forall q\in \overline{O_p^{w,y}}.
\end{equation}
We further define 
$O_p := \bigcap_{w\in S, y\in C_p}O^{w,y}_p.$ 

We also consider another (not necessarily open) cover $\set{K_p:p\in\cP}$ of $\cP$. For any fixed $p\in\cP$, if $C_p = S$, i.e. the entire state space is communicating, then let $K_p = \cP$. On the other hand, if $C_p^c\neq \varnothing$, we construct $K_p$ as in the following lemma.

\begin{lemma}\label{lemma:construct_K_p}
Assume the adversary is weakly communicating, $\cP$ is S-rectangular, and $\cQ$ and $\cP$ are compact. Then, for each $p\in\cP$ with $C_p^c\neq \varnothing$, there exists an open neighborhood $K_p$ of $p$ such that
\begin{equation}\label{eqn:wc_adv_K_eval_bd}
0\leq \sup_{\Delta:S\to\cQ} e^\top (I-M^{p}_{\Delta,q})^{-1} e \leq \sup_{\Delta:S\to\cQ} e^\top (I-M^{p}_{\Delta,p})^{-1} e + 1 < \infty,\quad \forall q\in \overline{K_p},
\end{equation}
where both suprema are attained. In this expression, $M^{p}_{\Delta,q}$ is the principal submatrix of $q_\Delta$ on $C^c_p$ defined by $M^{p}_{\Delta,q}(s,s') = q_\Delta(s'|s)$ for $s,s'\in C^c_{p}$, and $e$ denotes the all-ones vector in $\R^{|C_p^c|}$.
\end{lemma}

\begin{proof}[Remarks on the proof of Lemma \ref{lemma:construct_K_p}]
The proof of Lemma \ref{lemma:construct_K_p} follows from the same argument as that of Lemma \ref{lemma:construct_K_Delta}. In particular, the invertibility of $I-M_{\Delta,p}^p$ follows from the definition of $M_{\Delta,p}^p$ being the transient part of $p_{\Delta}$; the continuity of $(\eta,q)\ra (I-M_{\eta,q}^p)\inv$ within $\set{S\ra\cQ}\times K'_p$, where $K_p'$ is an open neighborhood of $p$, follows from the continuity of 
$$q\ra \min_{\Delta :S\ra\cQ} \abs{\det(I-M_{\Delta,q}^p)}$$ and the invertibility of $I-M_{\Delta,p}^p$ for all $\Delta:S\ra\cQ$. 

These properties imply the finiteness and continuity  of 
$$q\ra \max_{\Delta:S\ra\cQ}e^\top(I-M_{\Delta,q}^p)\inv e$$
within some open neighborhood of $p$, implying Lemma \ref{lemma:construct_K_p}. 
\end{proof}

As in the proof of Theorem \ref{thm:wc_ctrl}, we have defined $O_p$ and $K_p$ for all $p\in\cP$. With these constructions, we define
$$G_p := O_p \cap K_{p}.$$
Note that when $C_p^c = \varnothing$, then $G_p = O_p\ni p$ is non-empty and open. When $C_p^c \neq \varnothing$, both $O_p$ and $K_p$ are open neighborhoods of $p$. Therefore, $\set{G_p:p\in\cP}$ is an open cover of $\cP$. Since $\cP$ is compact, there exists a finite subcover $\set{G_{p}:p\in B'}$ where $B' :=\set{p_1,\ldots,p_{|B'|}}\subseteq \cP$ is a finite subset. 

With this construction, it is clear that the proof of Lemma \ref{lemma:wc_exists_adv_lower_bd_prob} carries over, and we have the following result. 

\begin{lemma}\label{lemma:wc_exists_adv_lower_bd_prob}
Under the assumptions of Theorem \ref{thm:wc_adv}, there exists $\delta>0$ such that, for any stationary adversary policy $p\in\cP$ with $p\in G_{p_k}$ and any $y\in C_{p_k}, w\in S$, there exists $\Delta:S\to\cQ$ and $N \leq |S|$ such that $p_\Delta^{N}(y|w)\geq \delta$.
\end{lemma}

Again, note that following the same argument in Remark \ref{rmk:cmpct_P_not_needed}, if we replace $G_{p_k}$ with $O_{p_k}$, we do not need the compactness of $\cQ$ to show Lemma \ref{lemma:wc_exists_adv_lower_bd_prob}. Moreover, as Lemma \ref{lemma:wc_ctrl_exists_adv_lower_bd_prob} implies Lemma \ref{lemma:wc_bd_E_hitting_time}, Lemma \ref{lemma:wc_exists_adv_lower_bd_prob} implies the following expected hitting time bound. 

\begin{lemma}\label{lemma:wc_adv_bd_E_hitting_time}
Under Lemma \ref{lemma:wc_exists_adv_lower_bd_prob}, there exists $\delta' > 0$ s.t. for any stationary adversary policy $p\in\cP$ with $p\in G_{p_k}$ and $y\in C_{p_k}$, there exists $\Delta:S\to\cQ$ such that 
$$\max_{w\in S}E^{\Delta,p}_w \tau_y \leq \frac{|S|}{\delta'},$$
where $\tau_y = \inf\set{t\geq 0: X_t =y}$. 
\end{lemma}

\begin{proof}[Proof of Lemma \ref{lemma:wc_adv_bd_E_hitting_time} based on the proof of Lemma \ref{lemma:wc_bd_E_hitting_time}]
Fix $p\in\cP$ and $y\in S$. By Lemma \ref{lemma:wc_exists_adv_lower_bd_prob}, for any $w\in S$ there exist a stationary controller $\Delta:S\ra\cQ$ and $N\le |S|$ such that 
\[
P^{\Delta,p}_w(\tau_y \le |S|) \ge p_\Delta^N(y|w) \ge \delta .
\]

As in the proof of Lemma \ref{lemma:wc_bd_E_hitting_time}, it suffices to  show that for some $\delta'>0$ independent of $y$, there is a controller $\eta:S\ra\cQ$ such that 
$$ \min_{w\in S}P_w^{\eta,p}(\tau_y\le |S|)\geq \delta'.$$
That is, the choice of controller can be made independent of the initial state $w$.

To this end, mirroring the proof of Lemma \ref{lemma:wc_bd_E_hitting_time}, let us define $\eta=\eta^{y,p}$ algorithmically as follows.

We initialize by assigning $\eta(\cd|y)=\phi$ for an arbitrary $\phi\in\cQ$. Then let $V=\set{y}$ be the assigned states, and $V^c=S\setminus V$ the unassigned states.  
\begin{enumerate}
    \item Choose any unassigned state $s_0\in V^c$. By Lemma \ref{lemma:wc_exists_adv_lower_bd_prob}, there exists $\Delta:S\ra\cQ$ and $N\le |S|$ such that $p_\Delta^N(y|s_0)\ge \delta$. Following the argument in the proof of Lemma \ref{lemma:wc_bd_E_hitting_time}, this implies that there exists a non-repeating path $s_0\ra s_1\ra \dots \ra s_k=y$ with $k\leq N$ and $p_\Delta(s_{i+1}|s_i)\ge \delta|S|^{-N}$ for $i=0,\dots,k-1$.

    \item Let $j=\min\{i\ge 1:s_i\in V\}$ be the first index on the path already in $V$. For $i=0,\dots,j-1$ set $\eta(\cd|s_i):=\Delta(\cd|s_i)\in\cQ$. Then $p_\eta(s_{i+1}|s_i)\ge \delta|S|^{-N}$ for all such $i$. Moreover, since $s_j\in V$, there is a path from $s_j$ to $y$ already recorded with each edge at least $\delta|S|^{-N}$. Concatenating, we obtain a path from $s_0$ to $y$ with all edge probabilities bounded below by $\delta|S|^{-N}$. Record this path. 
    \item Update $V\la V\cup\{s_0,\dots,s_{j-1}\}$. 

\end{enumerate}

Repeat until $V=S$.

The algorithm terminates in at most $|S|$ iterations, producing $\eta\in\{S\to\cQ\}$, since we always assign $\eta(\cd|s)\in\cQ$. Moreover, it produces a directed graph whose edges correspond to positive transition probabilities at least $\delta|S|^{-N}$ under $p_\eta(\cd|\cd)$, ensuring that every state can reach $y$ in at most $|S|$ steps. Therefore,
\[
\min_{w\in S}P_w^{\eta,p}(\tau_y\le |S|)\ge (\delta|S|^{-N})^{|S|} \ge (\delta|S|^{-|S|})^{|S|} =:\delta' >0.
\]

This and the Markov renewal argument in the proof of Lemma \ref{lemma:wc_bd_E_hitting_time} imply Lemma \ref{lemma:wc_adv_bd_E_hitting_time}. 
\end{proof}

\subsection*{Decomposing $\spnorm{v'_\gamma}$}
First, note that $v'_\gamma$ solves \eqref{eqn:discounted_inf_sup_eqn}. For each $\epsilon > 0$, there exist $p_\epsilon\in \cP$ that is $\epsilon$-optimal in the following sense
$$v'_\gamma(s) = \inf_{p\in \cP} \sup_{\pi\in \PiH}v_\gamma^{\pi,\kappa}(s)  \leq \sup_{\pi\in\PiH} v_\gamma^{\pi,p_\epsilon}(s) + \epsilon, \quad \forall s\in S.$$
Since $\set{G_{p}:p\in B'}$ covers $\cP$, there is $p_k\in B'$ s.t. $p_\epsilon\in G_{p_k}$. 

Moreover, by the stationary optimality of MDPs, there exists $\Delta_\epsilon:S\ra\cQ$ such that
$$\sup_{\pi\in\PiH} v_\gamma^{\pi,p_\epsilon}(s)\leq  v_\gamma^{\Delta_\epsilon,p_\epsilon}(s) + \epsilon, \quad \forall s\in S.$$

Similar to the decomposition for $\spnorm{v^*_\gamma}$, we let $$s_\vee\in\argmax_{s\in S}v_\gamma'(s), \quad s_\wedge\in \argmin_{s\in S}v'_\gamma(s),\quad\text{and}\quad y_\vee\in\argmax_{s\in G_{p_k}}v'_\gamma(s). $$ Then, we split $\spnorm{v'_\gamma}$ as follows: 
\begin{equation}\label{eqn:to_use_wc_adv_split}
\begin{aligned}
    \spnorm{v_\gamma'} &= v_\gamma'(s_\vee) - v_{\gamma}'(s_\wedge)\\
    &\leq \sup_{\pi\in\PiH}v_\gamma^{\pi,p_\epsilon}(s_\vee) - \sup_{\pi\in\PiH}v_\gamma^{\pi,p_\epsilon}(s_\wedge) +\epsilon\\
    &\leq v_{\gamma}^{\Delta_\epsilon,p_\epsilon}(s_\vee)  - \sup_{\pi\in\PiH}v_\gamma^{\pi,p_\epsilon}(s_\wedge) + 2\epsilon\\
    &\leq \underbrace{v_{\gamma}^{\Delta_\epsilon,p_\epsilon}(s_\vee) - v_\gamma^{\Delta_\epsilon,p_\epsilon}(y_\vee)}_{\xi_2} + \underbrace{v_{\gamma}^{\Delta_\epsilon,p_\epsilon}(y_\vee) - v_\gamma^{\pi,p_\epsilon}(s_\wedge)}_{\xi_1} + 2\epsilon
\end{aligned}
\end{equation}
for any $\pi\in\PiH$. 

\subsection*{Upper-Bounding $\xi_1$}

Following the same construction as in the proof of Theorem \ref{thm:wc_ctrl} (but swapping the roles of controller and adversary), we bound $\xi_1$ by considering a two-phase history-dependent control policy as follows. 

Fix $y\in C_{p_k}$ where $p_\epsilon\in G_{p_k}$. By Lemma \ref{lemma:wc_adv_bd_E_hitting_time} there exists a stationary controller policy $\eta$ such that
$$\max_{w\in S}E^{\eta,p_\epsilon}_w \tau_y \leq \frac{|S|}{\delta'}$$
Thus, define a history-dependent two-phase controller $\pi=(\pi_0,\pi_1,\ds)$ by
\[
\pi_t(\cdot|h_t) = 
\begin{cases} 
\eta(\cdot|s_t), & \text{if } s_0,\ds,s_{t}\neq y,\\[2pt]
\Delta_\epsilon(\cdot|s_t), & \text{otherwise.}
\end{cases}
\]

Then, we have that for all $x\in S$
\begin{equation}\label{eqn:to_use_adv_vpi_hd_bd}
    \begin{aligned}
        v_{\gamma}^{\pi,p_\epsilon}(x)
        &= E^{\pi,p_\epsilon}_x \sum_{k=0}^{\tau_y-1} \gamma^k r(X_k,A_k) + E^{\pi,p_\epsilon}_x \gamma^{\tau_y}\sum_{k=\tau_y}^{\infty} \gamma^{k-\tau_y} r(X_k,A_k)\\
        &\geq  E^{\pi,p_\epsilon}_x \gamma^{\tau_y}E^{\pi,p_\epsilon}_x\sqbkcond{\sum_{k=\tau_y}^{\infty} \gamma^{k-\tau_y} r(X_k,A_k)}{\cH_{\tau_y}}\\
        &=  v_\gamma^{\Delta_\epsilon,p_\epsilon}(y)E^{\pi,p_\epsilon}_x \gamma^{\tau_y}.
    \end{aligned}    
    \end{equation}
where the last equality follows from the same argument as in \eqref{eqn:to_use_second_half_eq}. Therefore, we have that 
$$\begin{aligned}
    v_{\gamma}^{\Delta_\epsilon,p_\epsilon}(y) - v_\gamma^{\pi,p_\epsilon}(x) 
    &\leq v_{\gamma}^{\Delta_\epsilon,p_\epsilon}(y) - v_\gamma^{\Delta_\epsilon,p_\epsilon}(y)E^{\pi,p_\epsilon}_x \gamma^{\tau_{y}}\\
    &= (1-\gamma)v_\gamma^{\Delta_\epsilon,p_\epsilon}(y) E^{\pi,p_\epsilon}_x \frac{1-\gamma^{\tau_{y}}}{1-\gamma}\\
    &\leq E^{\pi,p_\epsilon}_x \tau_{y}
\end{aligned}$$
where the last inequality follows from $0\leq v_\gamma^{\Delta_\epsilon,p_\epsilon}(y)\le1/(1-\gamma)$ and $(1-\gamma^t)/(1-\gamma) = \sum_{k=0}^{t-1}\gamma^k\leq t$. 

On the other hand, mirroring \eqref{eqn:to_use_E_kappa_tau_y}, we have that 
$E^{\pi,p_\epsilon}_x\tau_y = E_x^{\eta,p_\epsilon}\tau_y$. Therefore, choosing $x = s_\wedge$, $y = y_\vee\in C_{p_k}$, and $\eta = \eta^{y_\vee,p_\epsilon}$ given by Lemma \ref{lemma:wc_adv_bd_E_hitting_time}, we conclude that
$$\xi_1\leq E^{\pi,p_\epsilon}_{s_\wedge} \tau_{y_\vee} = E^{\eta,p_\epsilon}_{s_\wedge} \tau_{y_\vee} \leq \frac{|S|}{\delta'}. $$
\subsection*{Upper-Bounding $\xi_2$}

Let $T_k:= \inf\set{t\geq0: X_t\in C_{p_k}}$. To bound $\xi_2$, note that
\[\begin{aligned}v_{\gamma}^{\Delta_\epsilon,p_\epsilon}(s_\vee) - v_\gamma^{\Delta_\epsilon,p_\epsilon}(y_\vee) &\leq E_{s_\vee}^{\Delta_\epsilon,p_\epsilon}\sum_{t=0}^{T_k-1}\gamma^t r(X_t,A_t) + E_{s_\vee}^{\Delta_\epsilon,p_\epsilon}\sum_{t=T_k}^{\infty}\gamma^t r(X_t,A_t)  - v_\gamma^{\Delta_\epsilon,p_\epsilon}(y_\vee) \\
&\stackrel{(i)}{=} E_{s_\vee}^{\Delta_\epsilon,p_\epsilon}\sum_{t=0}^{T_k-1}\gamma^t r(X_t,A_t) + E_{s_\vee}^{\Delta_\epsilon,p_\epsilon} E_{X_{T_k}}^{\Delta_\epsilon,p_\epsilon}\sum_{t=0}^{\infty}\gamma^t r(X_t,A_t)  - v_\gamma^{\Delta_\epsilon,p_\epsilon}(y_\vee) \\
&\stackrel{(ii)}{=} E_{s_\vee}^{\Delta_\epsilon,p_\epsilon}\sum_{t=0}^{T_k-1}\gamma^t r(X_t,A_t) + E_{s_\vee}^{\Delta_\epsilon,p_\epsilon} v_\gamma^{\Delta_\epsilon,p_\epsilon}(X_{T_k})  - v_\gamma^{\Delta_\epsilon,p_\epsilon}(y_\vee) \\
&\stackrel{(iii)}{\leq} E_{s_\vee}^{\Delta_\epsilon,p_\epsilon}\sum_{t=0}^{T_k-1}\gamma^t r(X_t,A_t)\\
&\leq E_{s_\vee}^{\Delta_\epsilon,p_\epsilon} T_k,
\end{aligned}\]
where $(i)$ applies the strong Markov property, $(ii)$ recalls the definition of $v_\gamma^{\Delta_\epsilon,p_\epsilon}$, and $(iii)$ follows from $y_\vee\in C_{p_k}$ achieving the argmax.  

Then, following the same argument as in Theorem \ref{thm:wc_ctrl} and the construction of $K_{p}$, we conclude that
$$\xi_2\leq \max_{k\leq|B'|}\sup_{p\in\overline{G_{p_k}},\; \Delta:S\ra\cQ} E_{s_\vee}^{\Delta,p}T_{k}\leq 1+\sum_{k\leq |B'|}\sup_{\Delta:S\ra\cQ}e^\top(I-M_{\Delta,p_k}^{p_k})^{-1}e < \infty.$$

\begin{remark}\label{rmk:wc_adv_comm_case}
If the adversary is communicating, then $C_p = S$. In this case, Lemma \ref{lemma:wc_adv_bd_E_hitting_time} holds with $y\in S$ arbitrary, and the term corresponding to transient states is vacuous, and $\xi_2\leq 0$. Thus, the compactness of $\cQ$ is not required.
\end{remark}

Combining the bounds for $\xi_1$ and $\xi_2$, we have shown that $\spnorm{v'_\gamma}$ is uniformly bounded. 
By the same argument for Theorem \ref{thm:existence_of_sol}, the uniform boundedness of $\spnorm{v'_\gamma}$ implies the existence of $(u',\alpha')$ solving the average-reward inf-sup Bellman equation \eqref{eqn:inf_sup_eqn_const_gain}. 
\end{proof}

\section{Proof of Theorem \ref{thm:suff_cond_for_swap}}\label{sec:proof:thm:suff_cond_for_swap}
\begin{proof}
We first show statement (1). 

Note that the mapping $$(\phi,p_s)\ra \sum_{a\in A,\;s'\in S }\phi(a)p(s'|s,a)[r(s,a) - \alpha' + u'(s')] = E_{\phi,p_s}[r(s,A_0) - \alpha' +u'(X_1)]$$
is bi-linear in $\phi$ and $p_s$. Since both $\cQ$ and $\cP_s, \; s\in S$ are convex and one of them is compact, by Sion's minimax Theorem (Corollary 3.3 in \citet{sion1958minmax}), 
we have that for all $s\in S$, $$\begin{aligned}
 u'(s) &=\inf_{p_s\in\cP_s}\sup_{\phi\in\cQ} E_{\phi,p_s}[r(s,A_0) - \alpha' +u'(X_1)]   \\
 &= \sup_{\phi\in\cQ} \inf_{p_s\in\cP_s}E_{\phi,p_s}[r(s,A_0) - \alpha' +u'(X_1)].
\end{aligned}$$
Hence, $(u',\alpha')$ solves \eqref{eqn:r_bellman_eqn_const_gain}. 

To prove the second statement, we define 
$$q'(s,a) = r(s,a) -\alpha' + \inf_{\nu\in\cP_{s,a}}\sum_{s'\in S}\nu(s')u'(s').$$ Notice that since $\set{\delta_a:a\in A}\subset\cQ$, we have \begin{equation}\label{eqn:to_use_up_lower_bd}\begin{aligned}
    \max_{a\in A}q'(s,a) &= \sup_{\phi\in\cQ} \sum_{a\in A}\phi(a)q'(s,a)\\
    &= \sup_{\phi\in\cQ}\sum_{a\in A}\phi(a)\inf_{\nu\in\cP_{s,a}}\sum_{s'\in S}\nu(s')[r(s,a) - \alpha'+u'(s')]\\
    &\stackrel{(i)}{=}\sup_{\phi\in\cQ}\inf_{p_s\in\cP_{s}}\sum_{a\in A,\,s'\in S}\phi(a)p(s'|s,a)[r(s,a) - \alpha'+u'(s')]\\
    &\leq u'(s)
\end{aligned} \end{equation}
where $(i)$ follows from SA-rectangularity that $\cP_s = \bigtimes_{a\in A}\cP_{s,a}$. 

On the other hand, 
\begin{equation}\label{eqn:to_use_up_upper_bd}
\begin{aligned}
    u'(s) &= \inf_{p_s\in\cP_{s}}\sup_{\phi\in\cQ}\sum_{a\in A,\,s'\in S}\phi(a)p(s'|s,a)[r(s,a) - \alpha'+u'(s')]\\
    &= \inf_{\nu_{1}\in\cP_{s,a_1}} \ds \inf_{\nu_{|A|}\in\cP_{s,a_{|A|}}}\sup_{\phi\in\cQ}\sum_{i=1}^{|A|}\phi(a_i)\sqbk{r(s,a_i) - \alpha'+\sum_{s'\in S}\nu_{i}(s')u'(s')}.
\end{aligned}
\end{equation}
For any $\delta > 0,$ choose $\nu'_i\in\cP_{s,a_i}$ so that $$\sum_{s'\in S}\nu'_i(s')u'(s') \leq \inf_{\nu\in\cP_{s,a_i}} \sum_{s'\in S}\nu(s')u'(s')+\delta.$$ Then, continue from \eqref{eqn:to_use_up_upper_bd}, we have that
$$\begin{aligned}u'(s)&\leq \sup_{\phi\in\cQ}\sum_{i=1}^{|A|}\phi(a_i)\sqbk{r(s,a_i) - \alpha'+\sum_{s'\in S}\nu'_{i}(s')u'(s')}\\
&\leq \delta + \sup_{\phi\in\cQ}\sum_{i=1}^{|A|}\phi(a_i)\sqbk{r(s,a_i) - \alpha'+\inf_{\nu\in\cP_{s,a_i}}\sum_{s'\in S}\nu(s')u'(s')}\\
&= \delta + \sup_{\phi\in\cQ}\sum_{a\in A}\phi(a)q'(s,a)\\
&= \delta + \max_{a\in A}q'(s,a).
\end{aligned}
$$
Since $\delta > 0$ can be arbitrary small, combining this with \eqref{eqn:to_use_up_lower_bd}, we conclude that $u'(s) = \max_{a\in A}q'(s,a)$ for all $s\in S$. 

Finally, we note that from \eqref{eqn:to_use_up_lower_bd},
$$u'(s) = \max_{a\in A}q'(s,a) = \sup_{\phi\in\cQ}\inf_{p_s\in\cP_{s}}\sum_{a\in A,\,s'\in S}\phi(a)p(s'|s,a)[r(s,a) - \alpha'+u'(s')].$$
Therefore, $(u',\alpha')$ solves \eqref{eqn:r_bellman_eqn_const_gain}, completing the proof. \end{proof}

\section{Proof of Corollary \ref{cor:overlapped_recurrent_class}}\label{sec:proof_cor_overlapped_recurrent_class}

We show the two claims of Corollary \ref{cor:overlapped_recurrent_class} by discussing how the overlap-connectedness assumptions are related to weak communication. 

\begin{lemma}[Overlap-Connected Controller]\label{lemma:overlap_ctrl}
If $\cR_{\Delta}$ is overlap-connected and all $\cP_s$, $s\in S$ are convex, then $\Delta$ is weakly communicating with $C_{\Delta} = \bigcup_{R\in\cR_{\Delta}}R. $
\end{lemma}
\begin{proof}[Proof of Lemma \ref{lemma:overlap_ctrl}] Let $C_{\Delta} = \bigcup_{R\in\cR_{\Delta}}R$. By the definition of $\cR_{\Delta}$, for each $R\in \cR_{\Delta}$, there exists $p^{(R)}$ such that $p^{(R)}_{\Delta}$ has $R$ as one of its closed recurrent classes. 

Let $\cR_{\Delta}(s):=\set{R\in \cR_{\Delta}:s\in R}$. Then, define $$q(\cd|s,\cd):=\begin{cases}\displaystyle\frac{1}{|\cR_{\Delta}(s)|} \sum_{R\in \cR_{\Delta}(s)}p^{(R)}(\cd|s,\cd), & \text{if } s\in C_{\Delta},\\
\displaystyle \text{any }p_s\in\cP_s, & \text{otherwise.}
\end{cases}$$
By S-rectangularity and convexity, $q\in\cP$. 

By the overlap-connectedness of $\cR_{\Delta}$, for all $s,s'\in C_{\Delta}$ there exists $R_0,\ds,R_k$ such that $s\in R_0$, $s'\in R_k$, and $R_i\cap R_{i+1}\neq \varnothing$ for all $0\leq i\leq k-1$. Then, pick an arbitrary element $y_{i+1}\in R_i\cap R_{i+1}$. 

Let $y_0 = s$ and $y_{k+1} = s'$. As for all $i$, $y_i$ and $y_{i+1}$ are contained in $R_i$, there exists a path of positive probability under $p^{(R_i)}_{\Delta}$ that connects $y_i$ and $y_{i+1}$. By the definition of $q$, $q_{\Delta}(\cd|s)\geq \frac{1}{|\cR_{\Delta}(s)|}p_\Delta^{(R_i)}(\cd|s)$ for all $s\in C_{\Delta}$. It follows that the path leads $y_i$ to $y_{i+1}$ has positive probability under $q$ for each $i$. Therefore, $s$ leads to $s'$ under $q_\Delta$; i.e. there exists $N\geq 1$ s.t. $q_{\Delta}^N(s'|s) > 0$. 

To verify weak communication, we also need to check that any $z\in C_{\Delta}^c$ is transient under all $p\in\cP$. This is straightforward as, by definition, $z\in C_{\Delta}^c$ implies that $z\notin R$ for any $R\in\cR_{\Delta}$; i.e. it is transient under all $p_{\Delta}$, $p\in\cP$. This proves Lemma \ref{lemma:overlap_ctrl}.
\end{proof}

Similarly, we also show that if $\cQ$ is convex, an overlap-connected $\cR_{p}$ will imply the weak communication of $p\in\cP$. 

\begin{lemma}[Overlap-Connected Adversary]\label{lemma:overlap_adv}
If $\cR_{p}$ is overlap-connected and $\cQ$ is convex, then $p$ is weakly communicating with $C_{p} = \bigcup_{R\in\cR_{p}}R. $
\end{lemma}
\begin{proof}[Proof of Lemma \ref{lemma:overlap_adv}]. Let $C_{p} = \bigcup_{R\in\cR_{p}}R$. By the definition of $\cR_{p}$, for each $R\in \cR_{p}$, there exists $\Delta_R$ such that $p_{\Delta_R}$ has $R$ as one of its closed recurrent classes. 

Let $\cR_{p}(s):=\set{R\in \cR_{p}:s\in R}$. Then, define $$\eta(\cd|s):=\begin{cases}\displaystyle\frac{1}{|\cR_{p}(s)|} \sum_{R\in \cR_{p}(s)}\Delta_R(\cd|s), & \text{if } s\in C_{p},\\
\displaystyle \text{any }\mu\in \cQ, & \text{otherwise.}
\end{cases}$$
By convexity of $\cQ$, $\eta\in \set{S\ra\cQ}$. 

By the overlap-connectedness of $\cR_{p}$, for all $s,s'\in C_{p}$ there exists $R_0,\ds,R_k$ such that $s\in R_0$, $s'\in R_k$, and $R_i\cap R_{i+1}\neq \varnothing$ for all $0\leq i\leq k-1$. Then, pick an arbitrary element $y_{i+1}\in R_i\cap R_{i+1}$. 

Let $y_0 = s$ and $y_{k+1} = s'$. As for all $i$, $y_i$ and $y_{i+1}$ are contained in $R_i$, there exists a path of positive probability under $p_{\Delta_{R_i}}$ that connects $y_i$ and $y_{i+1}$. By the definition of $\eta$, 
$$p_{\eta}(\cd|s) = \sum_{a\in A}\eta(a|s)p(\cd|s,a)\geq \frac{1}{|\cR_{p}(s)|}\sum_{a\in A}\Delta_{R_i}(a|s)p(\cd|s,a) = \frac{1}{|\cR_{p}(s)|} p_{\Delta_{R_i}}(\cd|s)$$ for all $s\in C_{p}$. It follows that the path leads $y_i$ to $y_{i+1}$ has positive probability under $p_\eta$ for each $i$. Therefore, $s$ leads to $s'$ under $p_\eta$; i.e. there exists $N\geq 1$ s.t. $p_{\eta}^N(s'|s) > 0$. 

To verify weak communication, we also need to check that any $z\in C_{p}^c$ is transient under all $p\in\cP$. This is straightforward as, by definition, $z\in C_{p}^c$ implies that $z\notin R$ for any $R\in\cR_{p}$; i.e. it is transient under all $p_{\Delta}$, $\Delta:S\ra\cQ$. This proves Lemma \ref{lemma:overlap_adv}. 
\end{proof}

\begin{proof}[Proof of Corollary \ref{cor:overlapped_recurrent_class}]
By Lemma \ref{lemma:overlap_ctrl}, the convexity and overlap-connectedness assumption in the first statement of Corollary \ref{cor:overlapped_recurrent_class} ensures that every $\Delta:S\to\cQ$ is weakly communicating. Combined with the compactness of $\cQ$ and of each $\cP_s$ for $s\in S$, the assumptions of Theorem \ref{thm:wc_ctrl} are satisfied. Hence, \eqref{eqn:r_bellman_eqn_const_gain} admits a solution.  

Similarly, by Lemma \ref{lemma:overlap_adv}, the convexity and overlap-connectedness assumption in the latter statement implies that every $p\in\cP$ is weakly communicating. Together with the compactness of $\cQ$ and of each $\cP_s$, this verifies the assumptions of Theorem \ref{thm:wc_adv}. Therefore, \eqref{eqn:inf_sup_eqn_const_gain} admits a solution.  
\end{proof}

\section{Proofs for Section \ref{sec:HD_S}}

\subsection{Proof of Proposition \ref{prop:hd_eps_opt_policy}}\label{sec:proof_prop_RL_policy}
\begin{proof}
First, notice that 
\begin{equation}\label{eqn:to_use_415}
\begin{aligned}
0&\leq \underline{\alpha}(\mu,\PiH,\KS) - \inf_{\kappa\in\KS}\underline{\alpha} (\mu,\piRL,\kappa)  \\
&\stackrel{(i)}{\leq} \inf_{\kappa\in\KS}\sup_{\pi\in\PiH}\underline{\alpha}(\mu,\pi,\kappa) - \inf_{\kappa\in\KS}\underline{\alpha} (\mu,\piRL,\kappa)\\
&\stackrel{(ii)}{=} \inf_{\kappa\in\KS}\sup_{\pi\in\PiS}\underline{\alpha}(\mu,\pi,\kappa) - \inf_{\kappa\in\KS}\underline{\alpha} (\mu,\piRL,\kappa) \\
&\leq \sup_{\kappa\in\KS}\abs{\sup_{\pi\in\PiS}\underline{\alpha}(\mu,\pi,\kappa) - \underline{\alpha} (\mu,\piRL,\kappa)}\\
&\stackrel{(iii)}{=} \sup_{\kappa\in\KS}\abs{\alpha^*_\kappa- \underline{\alpha} (\mu,\piRL,\kappa)}\\
\end{aligned}
\end{equation}
where $(i)$ follows from weak duality and $(ii)$ uses the optimality of $\PiS$ for classical MDPs (see \citet{puterman2014MDP}). For $(iii)$, note that since $p\in \cP$ is weakly communicating, by the standard results from classical MDPs (also see \citet{puterman2014MDP}), we have that for each $\kappa\in\KS$, there exists an optimal deterministic Markov time-homogeneous policy $\Delta_\kappa$ that achieves an optimal average-reward $\alpha^*_\kappa$.

On the other hand, Algorithm 2 in \citet{zhang2023amdp_regret} and the regret bound therein imply that for any weakly communicating MDP and parameter $\epsilon >0$, there exists a policy $\piRL$ that uses only deterministic actions so that for any $\kappa\in\KS$, w.p. at least $1-\epsilon$
$$\sum_{t=0}^{n-1}[\alpha_\kappa^*-r(X_t,A_t)]= \widetilde O (\spnorm{h_\kappa^*}\sqrt{n})$$
for all sufficiently large $n$, where $\widetilde O(\cd)$ suppress the dependence on $\log n$ and $\log(1/\epsilon)$. 
This implies that $$0\leq \alpha_\kappa^* - E_{\mu}^{\piRL,\kappa}\frac{1}{n}\sum_{t=0}^{n-1}r(X_t,A_t)= \widetilde O \crbk{\frac{\spnorm{h_\kappa^*}}{\sqrt{n}}}(1-\epsilon) + \epsilon.$$
Hence, we have that 
\begin{equation}\label{eqn:to_use_431}
    \begin{aligned}
        0 &\le \alpha^*_\kappa -\overline{\alpha} (\mu,\piRL,\kappa) \\
        &\leq \alpha^*_\kappa -\underline{\alpha} (\mu,\piRL,\kappa) \\
        &= \limsup_{n\ra\infty}\crbk{\alpha^*_\kappa -  E_{\mu}^{\pi,\kappa}\frac{1}{n}\sum_{t=0}^{n-1}r(X_t,A_t) }\\
        &\leq \epsilon
    \end{aligned}
\end{equation}

Since $\piRL$ only uses deterministic actions and $\set{\delta_a:a\in A}\subseteq \cQ$, $\piRL\in \PiH(\cQ)$. Therefore, going back to \eqref{eqn:to_use_415}, we have that \[\begin{aligned}
 0&\leq \underline{\alpha}(\mu,\PiH,\KS) - \inf_{\kappa\in\KS}\underline{\alpha} (\mu,\piRL,\kappa)   \\
 &\leq \inf_{\kappa\in\KS}\sup_{\pi\in\PiS}\underline{\alpha}(\mu,\pi,\kappa) - \inf_{\kappa\in\KS}\underline{\alpha} (\mu,\piRL,\kappa)\\
 &\leq \sup_{\kappa\in\KS}\abs{\alpha^*_\kappa- \underline{\alpha} (\mu,\piRL,\kappa)} \\
 &\leq \epsilon.  
\end{aligned}\]

Finally, to conclude the proposition, we note that if $\underline\alpha$ is replaced by $\overline{\alpha}$, the derivation in \eqref{eqn:to_use_415} is still valid. This, coupled with \eqref{eqn:to_use_431}, yields the limsup version of Theorem \ref{prop:hd_eps_opt_policy}. 
\end{proof}

\subsection{Proof of Theorem \ref{thm:value_HD_S}}
\begin{proof}
By Proposition \ref{prop:hd_eps_opt_policy}, we have that for any $\epsilon > 0$, 
$$ -\epsilon \leq \underline{\alpha}(\mu,\PiH,\KS) - \inf_{\kappa\in\KS}\sup_{\pi\in\PiS}\underline{\alpha}(\mu,\pi,\kappa)  \leq \epsilon.$$
Also, by Markov optimality in classical MDPs \citep{puterman2014MDP}, $\sup_{\pi\in\PiS}\underline{\alpha}(\mu,\pi,\kappa) = \sup_{\pi\in\PiH}\underline{\alpha}(\mu,\pi,\kappa)$.
Since $\epsilon$ can be arbitrarily small, these inequalities imply the liminf version of the first claim Theorem \ref{thm:value_HD_S}. 
The same argument holds when $\underline{\alpha}$ is replaced with $\overline{\alpha}$.

To show the second claim, we note that the same argument as in the proof of Theorem \ref{thm:general_dpp} will imply that the solution $\alpha'$ to \eqref{eqn:inf_sup_eqn_const_gain} is the optimal average-reward $$\alpha'=\inf_{\kappa\in\KH}\sup_{\pi\in\PiH}\underline\alpha(\mu,\pi,\kappa) = \inf_{\kappa\in\KS}\sup_{\pi\in\PiH}\underline\alpha(\mu,\pi,\kappa)=\inf_{\kappa\in\KS}\sup_{\pi\in\PiS}\underline\alpha(\mu,\pi,\kappa).$$ 
So, $\alpha'$ corresponds to the inf-sup control value, which is shown to be equal to $\underline{\alpha}(\mu,\PiH,\KS)$. The same holds true if $\underline\alpha$ is replaced by $\overline{\alpha}$. 
\end{proof}

\subsection{Proof of Corollary \ref{cor:iff}}
\begin{proof}
    By Theorem \ref{thm:wc_ctrl} and \ref{thm:wc_adv}, solutions $(u^*,\alpha^*)$ and $(u',\alpha')$ to \eqref{eqn:r_bellman_eqn_const_gain} and \eqref{eqn:inf_sup_eqn_const_gain} exists under the assumptions of Corollary \ref{cor:iff}. Hence, by Theorem \ref{thm:general_dpp} and \ref{thm:value_HD_S}, $$\underline{\alpha}(\mu,\PiH,\KS) \;=\; \overline{\alpha}(\mu,\PiH,\KS) \;=\; \alpha',$$
    while 
    $$\underline{\alpha}(\mu,\PiS,\KS) \;=\; \overline{\alpha}(\mu,\PiS,\KS) \;=\; \alpha^*.$$
    This implies the corollary. 
\end{proof}

\end{document}